\documentclass[final,12pt,times,english,twoside,reqno]{amsart}
\usepackage{amsmath,amssymb,mathrsfs}
\usepackage{verbatim,setspace,enumerate}
\usepackage{soul,cite,lineno,cancel,ulem,graphicx}
\newcommand\smallbullet{%
    \raisebox{-0.20ex}{\scalebox{1.8}{$\cdot$}}%
}
\usepackage{hyperref,enumitem}
\usepackage{mathtools}
\usepackage{geometry}
\def\G{\Gamma}

\def\cal#1{\mathcal{#1}}
\newtheorem{theorem}{Theorem}[section]

\newtheorem{corollary}[theorem]{Corollary}
\newtheorem{remark}[theorem]{Remark}
\newtheorem{definition}[theorem]{Definition}
\newtheorem{example}[theorem]{Example}

\newcommand{\mg}{\marginpar}

\newcommand{\Addresses}{{
  \bigskip
  \footnotesize
    $^{1}$ $^{2}$  University ''Alexandru Ioan Cuza'', Faculty of Mathematics,
  Bd. Carol I, No. 11, Ia\c{s}i, 700506, ROMANIA,
  email:   croitoru@uaic.ro, gavrilut@uaic.ro
\\
    	$^{3}$  Petroleum-Gas University of Ploie\c{s}ti, Department of Computer Science,
      Information Technology, Mathematics and Physics,
         Bd. Bucure\c{s}ti, No. 39, Ploie\c{s}ti 100680, ROMANIA\\
          email: emilia.iosif@upg-ploiesti.ro
\\
$^{4}$   Department of Mathematics and Computer Sciences,
     University of  Perugia -- 1, Via Vanvitelli - 06123, Perugia, ITALY,
     email: anna.sambucini@unipg.it  	
}}
\title[A survey  on the Riemann-Lebesgue integrability in non-additive ...]{
A survey  on the Riemann-Lebesgue integrability in non-additive setting}
\author{ Anca Croitoru$^{1}$,  Alina Gavrilu\c{t}$^{2}$, Alina Iosif$^{3}$,  and Anna Rita Sambucini$^{4}$}
\begin{document}
\begin{abstract}
We present some results regarding Riemann-Lebesgue integral of a vector(real resp.) function relative to
an arbitrary non-additive set function. Then these results
are generalized to the case of Riemann–Lebesgue integrable interval-valued multifunctions.
\end{abstract}
%===============================
%===============================
\date{\today}

%=======================

\maketitle

\tableofcontents

\section{Introduction}
The theory of non-additive set functions and nonlinear integrals has become an important tool
in many domains such as: potential theory, subjective
evaluation, optimization, economics, decision making, data mining, artificial intelligence, accident rates estimations
(e.g. \cite[]{CCS,CMP,Gal,GZ,MM,Milne,Pap1,Pat,sh,SC,book,torra,W}).
In the literature several methods of integration for (multi)functions
based on extensions of the Riemann and Lebesgue integrals have been introduced and studied (see for example, \cite[]{Bir,BM,bcs2014,bmis,BS1,BS,bs2023,CCGS,CDPMS1,CDPMS2,CS1,cr2005,CR,CG,DM,dp-p,dmm16,PiMMS,dp2,PMS,FMNR,F,GP,Ga,Gav,GIC,G,M,Me,Pap,PGC}). In this context,
Kadets and Tseytlin \cite{K} have introduced the
absolute Riemann-Lebesgue $|RL|$ and unconditional Riemann-Lebesgue RL integrability, for Banach
valued functions with respect to countably additive measures. According to \cite{K}, in finite measure
spaces, the Bochner integrability implies $|RL|$ integrability, which is stronger than RL integrability, that
implies Pettis integrability. Contributions in this area are given in
\cite[]{ccgs16,cmn,cr2005,danilo,cgis2022,cgis2021,PMS,K2002,P,Po,Po1,PC,PS1,R1,R2,R3,flores,Sam,Si,SB,Spa,Yin}.

Interval Analysis, as particular case of Set-Valued Analysis, was introduced by Moore \cite[]{moore66},
motivated by its applications in computational mathematics (i.e. numerical analysis).
The interval-valued multifunctions and multimeasures are involved in various applied sciences,
as statistics, biology, theory of games, economics, social sciences and
software. Also they are used for example in signal and image processing,
since the discretization of a continuous signal causes
different sources of uncertainty and ambiguity, as we can see in
\cite{CSV3,CSV2,mesiar,TMV1,lopez,V1,WD1}.

In this chapter, we study the Riemann-Lebesgue integral with respect
to an arbitrary set function, not necessarily  countably additive.
We present some classical  properties of the integral together with
some relationships between this Riemann-Lebesgue integral and some other integrals known in the literature as the
Birkhoff simple and the Gould integrals. We describe also some convergence theorems (e.g. Lebesgue type convergence theorems, Fatou type theorem) for sequences of Riemann-Lebesgue integrable
functions. Then these results are extended to the case of Riemann-Lebesgue integrable
interval-valued multifunctions.\\
    The chapter is organized as follows. In section 2 we present some preliminaries which are necessary in what follows.
In section 3 we
introduce the Riemann Lebesgue integral and we present
different properties for the RL integral of a real function with respect to a non-additive set function. Also we establish some comparative results among this Riemann-Lebesgue integral, the Birkhoff simple and the Gould integrals. Then some convergence theorems for sequences of Riemann-Lebesgue integrable are pointed out.  Section 3 ends  with  a H\"{o}lder, Minkowski type  inequalities for  Riemann-Lebesgue integrable functions.
The results of Section 3 are then applied in  Section 4 to the case of
 interval-valued functions and set functions. 
We highlight that
 we present  some generalizations of the Riemann-Lebesgue integral in two directions: the non-additive setting and the interval-valued setting. Some difficulties arise in these approaches: -the techniques of the classical measure theory can no longer be used; -the set spaces are more difficult to use.
The motivation of it is mainly due the fact that, though finite or countable additivity is a fundamental concept in measure theory, it can be useless in some modeling problems of decision making, data mining, economy, computer science, game theory, subjective evaluation, fuzzy logic
as shown in \cite{CMP,danilo,mesiar,TMV1,LMP,LY,lopez,Pap1,PS,book,torra}. 
% For example, the efficiency of a number of persons doing some work together is not equal to the sum of the efficiencies of each person doing his own fragment of the work.\\

%=============================================
\section{Definitions and basic facts}
Suppose $S$ is a nonempty set and $\mathcal{C}$ a $\sigma$-algebra of subsets of $S$, while
$\mathcal{P}(S)$ denotes the family of all subsets of $S$.
For every $E\subset S$, as usual, let $E^c=S\setminus E$ and let $\chi_{E}$ be the characteristic function of $E$.

Following \mbox{\rm (\cite[Definition 1]{ccgis})}
A finite (countable, resp.) partition of $S$ is a finite
(countable, resp.) family of nonempty sets $P=\{E_{i}\}_{i=1}^{n}$ ($\{E_{n}\}_{n\in \mathbb{N}}$, resp.) $\subset \mathcal{C}$
such that $E_{i}\cap E_{j}=\emptyset ,i\neq j$ and 
$\bigcup
\limits_{i=1}^{n}E_{i}= S$ ($\bigcup\limits_{n\in \mathbb{N}}E_{n}= S$,
resp.).
\begin{itemize}
\item  If $P$ and $P^{\prime }$ are two
partitions of $S$, then $P^{\prime }$ is said to be \textit{finer than} $P$,
\begin{eqnarray}\label{minore-ouguale}
\phantom{a} \,\,\,\,  P\leq P^{\prime } \,\, \mbox{ (or  } P^{\prime }\geq P \mbox{), if every set of }
P^{\prime } \mbox{   is included in some set of    } P.
\end{eqnarray}
\item  The \textit{common refinement} of two  finite or countable
partitions $P=\{E_{i}\}$ and $P^{\prime }=\{G_{j}\}$ is the partition $
P\wedge P^{\prime }:=\{E_{i}\cap G_{j}\}$.

\item   A countable  tagged partition of $S$ if a family $\{(E_{n},  s_{n}), n \in \mathbb{N}\}$
such that $(E_{n})_{n}$ is a countable partition of $S$ and $s_{n} \in E_{n}$ for every $n \in \mathbb{N}$.
\end{itemize}

All over this chapter, without any additional assumptions,
$\nu :\mathcal{C}\rightarrow [0, \infty)$ will be a set function, such that $\nu (\emptyset )=0.$

 As in \cite[Definitions 2 and 3]{ccgis},
$\nu :\mathcal{C}\rightarrow [0, \infty) $ is said to be:
\begin{itemize}
\item
 \textit{monotone} if $\nu (A)\leq \nu (B)$, for every  $ A,B\in\mathcal{C}$,
with $A\subseteq B$ (such non additive measures are called also capacities or fuzzy measures);
\item \textit{\ subadditive} if $\nu (A\cup B)\leq \nu (A)+ \nu (B),$
for every $A,B\in \mathcal{C}$, with $A\cap B=\emptyset ;$

\item  \textit{a submeasure } (in the sense of Drewnowski \cite{Dr})
if $\nu $ is monotone and subadditive;

\item $\sigma $\textit{-subadditive} if $\nu (A)\leq
\sum\limits_{n=0}^{+\infty }\nu (A_{n}),$ for every sequence of (pairwise
disjoint) sets\textit{\ }$(A_{n})_{n\in \mathbb{N}}\subset $ \textit{$%
\mathcal{C}$, }with $A=\bigcup\limits_{n=0}^{+\infty }A_{n}$.

\item \textit{finitely additive} if $\nu (A\cup B)= \nu(A)+ \nu(B),$
for every disjoint sets $A,B\in \mathcal{C};$
\item   \textit{$\sigma $-additive} if $\nu
(\bigcup\limits_{n=0}^{+\infty }A_{n})=\sum\limits_{n=0}^{+\infty} \nu (A_{n})$%
, for every sequence of pairwise disjoint sets $(A_{n})_{n\in \mathbb{N}%
}\subset \mathcal{C}$;

\item    \textit{order-continuous} (shortly, \textit{o-continuous}) if $%
\lim\limits_{n\rightarrow +\infty }\nu (A_{n})=0,$ for every decreasing
sequence of sets $(A_{n})_{n\in \mathbb{N}}\subset \mathcal{C}$, with
$\bigcap\limits_{n=0}^{+\infty}A_{n}= \emptyset$ (denoted by
$A_{n}\searrow \emptyset $);

\item    \textit{exhaustive }if $\lim\limits_{n\rightarrow +\infty }\nu
(A_{n})=0,$ for every sequence of pairwise disjoint sets $(A_{n})_{n\in
\mathbb{N}}\subset \mathcal{C}.$

\item
 {\it null-additive} if, for every $A, B \in \mathcal{C}$, $\nu(A \cup B) = \nu(A)$ when $\nu(B) = 0$.
\end{itemize}

Moreover
a set function  $\nu :\mathcal{C}\rightarrow [0, \infty)$ satisfies:
\begin{description}
\item[($\boldsymbol{\sigma}$)]
 the property $\boldsymbol{\sigma}$ if for every $\{E_{n}\}_{n}\subset \mathcal{C}$
with $\nu(E_{n})=0,$ for every $ n\in \mathbb{N}$ we have $\nu(\cup_{n=0}^{\infty}E_{n})= 0;$
\item[(E)] the condition {\bf (E)} if for every double sequence
$(B_{n}^{m})_{n, m\in \mathbb{N}^{*}}\subset \mathcal{C}$,
such that for every $m\in \mathbb{N}^{*}$, $B_{n}^{m}\searrow B^{m}\,  (n\rightarrow \infty)$ and
 $\, \nu(\cup_{m=1}^{\infty}B^{m})=0$,
there exist two increasing sequences $(n_{p})_p, (m_p)_p \subset \mathbb{N}$ such that
$\lim\limits_{k\rightarrow \infty}\nu(\bigcup_{p=k}^{\infty} \, B_{n_{p}}^{m_{p}})=0.$
\end{description}

The property $\boldsymbol{\sigma}$ is a consequence of the countable subadditivity and it will be needed in some of our results.
Observe that the condition {\bf (E)} was given, for example, in \cite{LY}, in order to give sufficient and necessary
conditions to obtain Egoroff's Theorem for suitable non additive measures. See also \cite{Pap2} for null additive set functions and related questions.
An example of a set function that satisfies the condition {\bf (E)} can be found in \cite[Example 3.3]{LY}.\\

%As in  (\cite[Definitions 2 and 3]{ccgis}
A property $(P)$ holds $\nu$-almost everywhere (denoted by
$\nu$-a.e.) if there exists $E \in \mathcal{C}$, with $\nu(E) = 0$, so that the property $(P)$ is valid on $S\setminus E.$
 A set $A\in \mathcal{C}$ is said to be an atom of a set function $\nu:\mathcal{C}\to [0, \infty)$ if
 $\nu(A)>0$ and for every $B\in \mathcal{C}$, with $B\subseteq A$, we have $\nu(B) = 0$ or $\nu(A\backslash B) = 0.$\\

We associate to $\nu:S\to [0, \infty)$ the following set functions.
\mbox{\rm (See \cite[Definition 4]{ccgis})}
\begin{itemize}
\item  The variation $\overline{\nu }$ of $\nu$ is the set
function $\overline{\nu }:\mathcal{P}(S)\rightarrow \lbrack 0,+\infty ]$
defined by
\[\overline{\nu }(E)=\sup \{\sum\limits_{i=1}^{n}\|\nu (A_{i})\|\},\]
for every $E\in \mathcal{P}(S)$, where the supremum is extended over all
finite families of pairwise disjoint sets $\{A_{i}\}_{i=1}^{n}\subset
\mathcal{C}$, with $A_{i}\subseteq E$, for every $ i\in\{1, \ldots, n\}$.
The set function  $\nu $ is said to be \textit{of finite variation}
(on $\mathcal{C}$) if $\overline{\nu }(S)<+\infty$.

\item
the semivariation
$\widetilde{\nu}$ of $\nu$ is the set function $:\mathcal{P}(S)\rightarrow \lbrack
0,+\infty ]$   defined for every $A\subseteq S$, by
\[
\widetilde{\nu }(A)=\inf \{\overline{\nu }(B);\,\, A\subseteq B,\,\, B\in \mathcal{C}
\}.\]
\end{itemize}

\begin{remark}\label{rem-var}
\rm
\mbox{\rm (\cite[Remark 1]{cgis2022})}
Let $\nu:\cal{C}\rightarrow [0,+\infty)$ be a non additive measure. Then
 $\overline{\nu}$ is monotone and super-additive on $\mathcal{P}%
(S)$, that is
\[\overline{\nu}(\bigcup_{i\in I}A_{i})\geq
\sum_{i\in I}\overline{\nu}(A_{i}),\]
 for every finite or countable partition $\{A_{i}\}_{i\in I}$ of $S$.
 If $\nu$ is finitely additive, then $\overline{\nu}(A)= \nu(A),$
for every $A\in \mathcal{C}.$
 If $\nu$ is subadditive ($\sigma$-subadditive,
resp.), then $\overline{\nu}$ is finitely additive ($\sigma $-additive, resp.).
Moreover, for every $\nu, \, \nu_1, \, \nu_2:\mathcal{C}\rightarrow \mathbb{R}$ and every $\alpha \in \mathbb{R}$,
\begin{itemize}
\item $  \overline{\nu_1 \pm \nu_2} \leq \overline{\nu_1} + \overline{\nu_2}$;
\qquad \quad   $\overline{\alpha \,\nu} = |\alpha| \overline{\nu}$.
\end{itemize}
\end{remark}

For all unexplained definitions, see for example \cite{CCGS,CCG2}.\\

Let $\mathscr{M}(S)$ be the set of all non negative submeasures on $(S, \mathcal{C})$.
Let $(\nu_n)_n \subset \mathscr{M}(S)$,
we will use the symbol $\nu_n \uparrow$ to indicate that $\nu_n \leq \nu_{n+1}$ for every $n \in \mathbb{N}$.\\

\begin{definition}\label{forte}
A sequence  $(\nu_n)_n \subset \mathscr{M}(S)$ setwise converges to  $\nu \in \mathscr{M}(S)$ if for every $A \in \mathcal{C}$
\begin{eqnarray}\label{setw}
 \lim_{n \to \infty} \overline{\nu_n - \nu}(A) = 0.
\end{eqnarray}
\end{definition}

\noindent
In  the $\sigma$-additive case, the setwise convergence
is given by $\lim_{n \to \infty} \nu_n(A) = \nu(A)$ for every $A \in \mathcal{C}$, see for example \cite{PiMMS}.
Since $|\nu_n(A)  - \nu(A)| \leq \overline{\nu_n - \nu}(A)$ for every $A \in \mathcal{C}$, the convergence given in Definition \ref{forte} implies the one of \cite{PiMMS}; the converse does not hold in general.
Neverthless, the two definitions coincide, from \cite[Remark 1]{ccgis}, if $\nu,\, \nu_n$, for all $n \in \mathbb{N}$, are finitely additive and non negative.\\

Finally, if  $S$ is a locally compact Hausdorff topological space, we denote by $\mathcal{K}$ the lattice of all compact subsets of
$S$, $\mathcal{B}$ the Borel $\sigma$-algebra (i.e.,the smallest $\sigma$-algebra containing $\mathcal{K}$) and $\mathcal{O}$
the class of all open sets.
\begin{definition} \rm
 A  set function $\nu: \mathcal{B} \to [0, \infty)$ is called regular if for every set $A\in \mathcal{B}$ and every $\varepsilon>0$
 there exist $K\in \mathcal{K}$ and $D\in \mathcal{O}$ such that $K\subseteq A\subseteq D$ and $\nu(D\setminus K)<\varepsilon.$
\end{definition}

%========================================
\section{The Riemann-Lebesgue integrability}
Let $X$ be a Banach space over $\mathbb{R}$. 
As in Kadets and Tseytlin \cite[Definition 4.5]{K} (for scalar functions) and Potyrala \cite[Definition 7]{Po}  and Kadets and Tseytlin \cite{K2002} (for vector functions),  we introduce the following definition:

\begin{definition}\label{def-int}
%\mg{3.1}
\mbox{\rm (\cite[Definition 5]{ccgis})}
\rm A vector function $f:S\to X$ is called
\textit{absolutely (unconditionally} resp.) \textit{Riemann-Lebesgue} ($|RL|$) ($RL$ resp.)
\textit{$\nu$-integrable} (on $S$) if there exists $a\in X$
such that for every $\varepsilon>0$, there exists a countable partition
$P_{\varepsilon}$ of $S$, so that for every countable partition
$P=\{A_{n}\}_{n\in \mathbb{N}}$ of $S$ with $P\geq P_{\varepsilon}$,
\begin{itemize}
\item
$f$ is bounded on every $A_{n}$, with $\nu(A_{n})>0$ and
\item for every $s_{n}\in A_{n}$, $n\in \mathbb{N}$,
the series $\sum_{n=0}^{+\infty}f(s_{n})\nu(A_{n})$ is absolutely (unconditionally resp.)
convergent and
\[\Big\Vert \sum_{n=0}^{+\infty}f(s_{n})\nu(A_{n})- a\Big\Vert<\varepsilon.\]
\end{itemize}
\end{definition}
The vector $a$ is called \textit{the absolute (unconditional)
Riemann-Lebesgue} $\nu$\textit{-integral of} $f$ \textit{on} $S$ and it
is denoted by ${\scriptstyle (|RL|)}\displaystyle{\int_S} f \, \mathrm{d}\nu$ \Big(${\scriptstyle (RL)}\displaystyle{\int_S} f \, \mathrm{d}\nu$ resp. \Big).\\

We denote by the symbol $|RL|^1_{\nu}(X)$ the class of all $X$-valued function that are $|RL|$ integrable with respect to $\nu$
and in an analogous way we denote the class of all functions that are $RL$
$\nu$-integrable.

\begin{remark}\label{rm-a} \rm \phantom{a}
(see \cite[Remark 2]{ccgis})
Obviously  if $a$ exists, then it is unique. Moreover, if $h$ is $|RL|$ $\nu$-integrable, then $h$ is $RL$ $\nu$-integrable and if
 $X$ is finite dimensional, then $|RL|$ $\nu$-integrability is equivalent to
$RL$ $\nu$-integrability. In this case, it is denoted by $RL$.

%If $X$ is finite dimensional, $RL$ $\nu$-integrability is called $\nu$-integrability for short and the integral is denoted by $\displaystyle{\int_S h \, \mathrm{d}\nu}$.\\
We remember also the following  in the countably additive case:
\begin{description}
\item [\ref{rm-a}.a)] Kadets and Tseytlin \cite{K} introduced the $|RL|$ $\nu$-integral and the $RL$ $\nu$-integral
for functions with values in a Banach space relative to a measure. They proved that if
$(S, \mathcal{C}, \nu)$ is a finite measure space, then the following implications hold:
%\\
 %Bochner integrability  $\Longrightarrow$
%$|RL|$ $\nu$-integrability $\Longrightarrow$
%$RL$ $\nu$-integrability $\Longrightarrow$
% Pettis integrability.

\[L^1_{\nu} (X) \subset |RL|^1_{\nu}(X) \subset RL^1_{\nu}(X) \subset P_{\nu} (X).\]
where $L^1_{\nu} (X),$ and $ P_{\nu} (X)$ denotes respectively the
Bochner and the Pettis  integrability.
\item [\ref{rm-a}.b)] If $X$ is a separable Banach space, then
%$|RL|$ $\nu$-integrability  $\Longrightarrow$   Bochner integrability
%and $RL$ $\nu$-integrability  $\Longrightarrow$   Pettis integrability.
\[L^1_{\nu} (X) = |RL|^1_{\nu}(X) \subset RL^1_{\nu}(X) = P_{\nu} (X.)\]
\item[\ref{rm-a}.c)] If $(S, \mathcal{C}, \nu)$ is a $\sigma$-finite measure space, then the Birkhoff integrability coincides
with $RL$ $\nu$-integrability (\cite{Po}).
\item [\ref{rm-a}.d)] If $h:[a, b]\rightarrow \mathbb{R}$ is Riemann integrable, then $h$ is $RL$-integrable (\cite[Corollary 17]{Po}).
The converse is not valid: for example the function $h:[0, 1]\rightarrow \mathbb{R}$, $h=\chi_{[0, 1]\cap \mathbb{Q}}$
is $RL$-integrable, but it is not Riemann integrable (\cite[Example 19]{Po}).
\end{description}
\end{remark}
%=========================================
%===============================================
\subsection{Some properties of $RL$ $\nu$-integrability}\label{RLint}
%\mg{RL-int}
In this section we present some results  contained in \cite{CCGS,ccgis}  regarding Riemann-Lebesgue integrability of vector
functions with respect to an arbitrary non-negative set function,
pointed out its remarkable properties.
We begin with a characterization of $|RL|$-integrability.
\begin{theorem}\label{3.1}
%\mg{3.4}
Let $g,h \in |RL|^1_{\nu} (X)$  and $\alpha, \beta \in\mathbb{R}$. Then:
\begin{description}
\item[(\ref{3.1}.a)]
If $h$  is $|RL|$ $\nu$-integrable on $S$, then  $h$ is $|RL|$ $\nu$-integrable on  every   $E\in
\mathcal{C}$ \mbox{\rm (\cite[Theorem 1.a]{ccgis})};
\item[(\ref{3.1}.b)]  $h$ is $|RL|$ $\nu$-integrable on  every   $E\in
\mathcal{C}$ if and only if
 $h\chi _{E}$ is $|RL|$ $\nu$-integrable on $S$.
 In this case, by  \mbox{\rm \cite[Theorem 1.b]{ccgis}},
 \[
{\scriptstyle (|RL|)}\int_{E}h\, \mathrm{d}\nu ={\scriptstyle (|RL|)}\int_S h\chi _{E}\, \mathrm{d}\nu .\]
 \end{description}
{\rm  (} The same holds for $RL$-integrability {\rm )}.
Moreover, by
\mbox{\rm \cite[Theorem 3]{ccgis},}
\begin{description}
\item[(\ref{3.1}.c)]  $\alpha g+\beta h \in|RL|^1_{\nu} (X)$  and
\begin{eqnarray*}
\displaystyle{{\scriptstyle (|RL|)}\int_S}(\alpha g +\beta h)\, \mathrm{d}\nu =
\alpha \cdot \displaystyle{{\scriptstyle (|RL|)}\int_S}g\, \mathrm{d}\nu
+\beta \cdot \displaystyle{{\scriptstyle (|RL|)}\int_S}h\, \mathrm{d}\nu,
\end{eqnarray*}
\item[(\ref{3.1}.d)] $h \in |RL|^1_{\alpha \nu}(X) $  for $\alpha \in
\lbrack 0,+\infty )$ and
\[
\displaystyle{{\scriptstyle (|RL|)}\int_S}h\, \mathrm{d}(\alpha \nu )=
\alpha \displaystyle{{\scriptstyle (|RL|)}\int_S}h\, \mathrm{d}\nu .
\]
\item[(\ref{3.1}.e)]
Suppose $h \in |RL|^1_{\nu_i}(X)$   for $i=1,2$. By
\mbox{\rm (\cite[Theorem 4]{ccgis})} $h \in |RL|^1_{\nu_1+\nu_2} (X)$  and
$$
{\scriptstyle (|RL|)}\int_S h \, \mathrm{d}(\nu_{1}+\nu_{2})=
{\scriptstyle (|RL|)}\int_S h \, \mathrm{d}\nu_{1}+{\scriptstyle (|RL|)}\int_S h \, \mathrm{d}\nu_{2}.$$

\end{description}
 Similar results also hold for the $RL$ $\nu$-integrability.
\end{theorem}

\begin{proof} We report here only the proofs of (\ref{3.1}.a) and (\ref{3.1}.b).\\
 Fix any $A\in \mathcal{C}$ and
denote by $J$ the integral of $h$ on $S$; then, fixed any $\varepsilon>0$,  there exists a partition $P_{\varepsilon}$ of $S$,
such that, for every finer partition $P':=\{A_{n}\}_{n\in \mathbb{N}}$ it is
$$\left\|\sum_{n=0}^{+\infty}h(t_n) \nu(A_n)-J \right\|\leq \varepsilon.$$
Now,
denote by $P_0$ any partition  finer than $P'$ and also finer than $\{A, S\setminus A\}$, and by $P_A$ the partition
of $A$ consisting of all the elements of $P_0$ that are contained in $A$.\\
 Next, let $\Pi_A$ and $\Pi'_A$ denote two
 partitions of $A$ finer than $P_A$, and extend them with a common partition of $S\setminus A$ (also with the
 same {\it tags}) in such a way that the two resulting partitions, denoted by $\Pi$ and $\Pi'$, are both finer
than $P'$. So,  if we denote by
\begin{eqnarray}\label{sigmap}
\sigma(h,\Pi):=\sum_{n=0}^{\infty}h(t_{n}) \nu(A_{n}), \quad A_n \in \Pi,
\end{eqnarray}
then
\[\|\sigma(h,\Pi)-\sigma(h,\Pi')\|\leq
\|\sigma(h,\Pi)-J\|+\|J-\sigma(h,\Pi')\|\leq 2\varepsilon.
\]
Now, setting:
\[
\alpha_1:=\sum_{I\in \Pi_A}h(t_I) \nu(I),\hskip.6cm \alpha_2:=\sum_{I\in \Pi'_A}h(t'_I) \nu(I), \hskip.6cm
\beta:=\sum_{I\in \Pi, I\subset A^c} h(\tau_I) \nu(I),
\]
(with obvious meaning of the symbols), one has
\[
2\varepsilon\geq\|\alpha_1 + \beta - (\alpha_2+\beta)\|=\|\alpha_1-\alpha_2\|.
\]
By the arbitrariness of $\Pi_A$ and $\Pi'_A$, this means that the sums $\sigma(h,\Pi_A)$ satisfy a Cauchy
principle in $X$, and so the first claim follows by completeness.
\\
  Now, let us suppose that $f$ is $|RL|$ $\nu$-integrable on
$A\in \mathcal{C}$.\\
 Then for every $\varepsilon >0$ there exists a partition $
P_{A}^{\varepsilon }\in \mathcal{P}_{A}$ so that for every partition $
P_{A}=\{B_{n}\}_{n\in \mathbb{N}}$ of $A$ with
$P_{A}\geq P_{A}^{\varepsilon}$ and for every $s_{n}\in B_{n},n\in \mathbb{N}$, we have

\begin{eqnarray}\label{1}
\Big\|\sum_{n=0}^{\infty}h(s_{n}) \nu (B_{n})-{\scriptstyle (|RL|)}\int_{A}h\, \mathrm{d}\nu \Big\|<\varepsilon.
\end{eqnarray}
Let us consider
$P_{\varepsilon }=P_{A}^{\varepsilon }\cup \{S\setminus A\}$,
which is a partition of $S$.
 If $P=\{A_{n}\}_{n\in \mathbb{N}}$ is a partition of $S$ with
$P\geq P_{\varepsilon }$, then without any loss of generality we may write
$
P=\{C_{n},D_{n}\}_{n\in \mathbb{N}}$ with pairwise disjoint $C_{n},D_{n}$
such that $A=\displaystyle\cup _{n=0}^{\infty }C_{n}$ and
$\displaystyle\cup_{n=0}^{\infty }D_{n}=S\setminus A.$
Now, for every $u_{n}\in A_{n},n\in \mathbb{N}$ we get by (\ref{1}):
\begin{eqnarray*}
&&\left\|\sum_{n=0}^{\infty}h\chi _{A}(u_{n}) \nu(A_{n})- {\scriptstyle (|RL|)}\int_{A}h \, \mathrm{d}\nu \right\|=\\
= && \left\|\sum _{n=0}^{\infty}h\chi _{A}(t_{n}) \nu(C_{n})+
\sum\limits_{n=0}^{\infty}h\chi _{A}(s_{n})\nu(D_{n})-{\scriptstyle (|RL|)}\int_{A} h\, \mathrm{d}\nu
\right\|=\\
=&&\left\|\sum _{n=0}^{\infty}h(t_{n}) \nu(C_{n})-{\scriptstyle (|RL|)}\int_{A} h\, \mathrm{d}\nu \right\|<\varepsilon ,
\end{eqnarray*}
where $t_{n}\in C_{n}, s_{n}\in D_{n},$ for every $n\in \mathbb{N},$
which says that $f\chi _{A}$ is $|RL|$ $m $%
-integrable on $S$ and $\displaystyle{{\scriptstyle (|RL|)}\int_S} h\chi _{A}\, \mathrm{d}\nu :={\scriptstyle (|RL|)}\int_{A}h \,\mathrm{d}\mu .$\\

Finally, suppose that $f\chi _{A}$ is $|RL|$ $\nu$-integrable on $S$. Then for every
$\varepsilon >0$ there exists $P_{\varepsilon }=\{B_{n}\}_{n\in \mathbb{N}}\in
\mathcal{P}$ so that for every
$P=\{C_{n}\}_{n\in \mathbb{N}}$ partition of $S$ with $P\geq P_{\varepsilon }$
and every $t_{n}\in C_{n},n\in \mathbb{N}$, we have
\begin{eqnarray}\label{2}
\left\|\sum_{n=0}^{\infty} h \chi _{A}(t_{n}) \nu (C_{n})-
{\scriptstyle (|RL|)}\int_S  h \chi_{A}\, \mathrm{d}\nu \right\|<\varepsilon.
\end{eqnarray}
Let us consider $P_{A}^{\varepsilon }=\{B_{n}\cap A\}_{n\in \mathbb{N}}$,
which is a partition of $A$. Let $P_{A}=\{D_{n}\}_{n\in \mathbb{N}}$
be an arbitrary partition of $A$ with $P_{A}\geq P_{A}^{\varepsilon }$ and
$P=P_{A}\cup \{S\setminus A\}.$ Then $P$ is a countable  partition finer than
$P_{\varepsilon }.$
 Let us take $t_{n}\in D_{n}, \, n\in \mathbb{N}$ and $s\in S\setminus A$. By (\ref{2}) we obtain
\begin{eqnarray*}
&&\left\|\sum_{n=0}^{\infty}h(t_{n})\nu(D_{n})-
\displaystyle{{\scriptstyle (|RL|)}\int_S} h \chi _{A}\, \mathrm{d}\nu \right\|=\\
= &&
\left\|\sum_{n=0}^{\infty} h \chi _{A}(t_{n}) \nu (D_{n})+
h \chi _{A}(s) \nu  (S\setminus A)-
\displaystyle{{\scriptstyle (|RL|)}\int_S} h \chi _{A}\, \mathrm{d}\nu \right\|<\varepsilon ,
\end{eqnarray*}
which assures that $f$ is $|RL|$ $\nu$-integrable on $A$ and
\[
{\scriptstyle (|RL|)}\int_{A} h \, \mathrm{d}\nu :={\scriptstyle (|RL|)}\int_S h \chi _{A}\, \mathrm{d}\nu.
\]
\end{proof}
In particular the $|RL|$ $\nu$-integrability with respect to a set function of finite variation
allows to obtain the following properties.
\begin{theorem}\label{upper-bound}
%\mg{upper-bound}
\mbox{\rm (\cite[Proposition 1, Theorems 2 and 5, Corollary 2]{ccgis})}
Let $\nu:S\to [0, \infty)$ be of finite variation.
If we suppose that  $h:S \to X$ is bounded
\begin{description}
\item [(\ref{upper-bound}.a)]    then  $h \in |RL|^1_{\nu} (X)$  and
	\[
	\Big\Vert {\scriptstyle (|RL|)}\int_S h \, \mathrm{d}\nu \Big\Vert
	\leq \sup\limits_{s\in S}\Vert h(s)\Vert \cdot \overline{\nu}(S).
	\]
\item[(\ref{upper-bound}.b)] If $h$ $=0$ $\nu$-a.e., then $h \in |RL|^1_{\nu} (X)$  and
	$\displaystyle{{\scriptstyle (|RL|)}\int_S}h\, \mathrm{d}\nu = 0$.
\end{description}
Moreover let  $g,h:S\rightarrow X$ be vector functions.
\begin{description}
\item[(\ref{upper-bound}.c)] If
 	$\sup\limits_{s\in S} \|g(s)-h(s)\|<+\infty $, $g \in |RL|^1_{\nu} (X)$  and $
	g=h \,\,\nu $-a.e., then $h\in |RL|^1_{\nu} (X)$  and
	\[ \displaystyle{{\scriptstyle (|RL|)}\int_S}g\, \mathrm{d}\nu
	=\displaystyle{{\scriptstyle (|RL|)}\int_S}h\, \mathrm{d}\nu .\]

\item[(\ref{upper-bound}.d)] If $g,h \in |RL|^1_{\nu} (X)$  then
\begin{eqnarray*}
\Big\Vert \displaystyle{{\scriptstyle (|RL|)}\int_S}g\, \mathrm{d}\nu - {\scriptstyle (|RL|)}\int_{S}h\, \mathrm{d}\nu \Big\Vert
\leq \sup_{s\in S}\Vert g(s)-h(s)\Vert\cdot \overline{\nu }(S).
\end{eqnarray*}
\end{description}
\end{theorem}

\begin{proof} We prove here (\ref{upper-bound}.b). From the boundedness of $h$ let
 $M\in \lbrack 0,\infty )$
so that $\|h(s)\|\leq M$, for every $s \in S$.
If $M=0$, then the conclusion is obvious.  Suppose $M>0.$ Let us denote
$A=\{s \in S : h(s)\neq 0 \}$. Since $h=0$ $\nu$-ae, we have
$\widetilde{\nu}(A)=0 $. Then, for every $\varepsilon >0$, there exists $B_{\varepsilon } \in
\mathcal{C}$ so that $A\subseteq B_{\varepsilon }$ and
$\overline{\nu}(B_{\varepsilon })<\varepsilon /M.$ Let us take the partition
$P_{\varepsilon }=\{C_{n}\}_{n\in \mathbb{N}}$ of $B_{\varepsilon}$,
and let $C_{0}=S\setminus B_{\varepsilon}$ and  add $C_0$ to $P_{\varepsilon }$.

Let  $P=\{A_{n}\}_{n\in \mathbb{N}}$  be an arbitrary partition of $S$ so that
$P\geq P_{\varepsilon }$.
 Without any loss of generality, we suppose that $P=\{D_{n},E_{n}%
\}_{n\in \mathbb{N}}\subset \mathcal{C},$
with pairwise disjoint sets $D_{n},E_{n}$ such that
\[ \bigcup_{n\in \mathbb{N}}D_{n}=C_{0} \qquad
\bigcup_{n\in\mathbb{N}}E_{n}=B_{\varepsilon }.\]
Let  $t_{n}\in D_{n}, s_{n}\in E_{n}$, for every $n\in \mathbb{N},$
%Let $u_{n}\in A_{n}$, for every $n\in \mathbb{N}$.
Then we can write

\begin{eqnarray*}
&&
 \big \Vert \sum\limits_{n=0}^{\infty}h(t_{n}) \nu(D_{n})+
\sum\limits_{n=0}^{\infty} h(s_{n}) \nu (E_{n})\big\Vert=
\|\sum\limits_{n=0}^{\infty}h(s_{n}) \nu (E_{n})\|\leq  \\
 &\leq&
 \sum\limits_{n=0}^{\infty}\|h(s_{n})\|\nu (E_{n})\leq
 M\cdot \overline{\nu}
(B_{\varepsilon })<\varepsilon,
\end{eqnarray*}
which ensures that $h$ is $|RL|$ $\nu$-integrable and $\displaystyle{{\scriptstyle (|RL|)}\int_T}h \, \mathrm{d}\nu =0.$
\end{proof}

%=============================
The next  theorem shows that the integral of a real function is monotone with respect to the
 integrands and to the set functions %with respect to the integration is considered,
(see \cite[Theorems 6 and 7]{ccgis}) in the following way.
\begin{theorem}\label{monotonicity}
%\mg{monotonicity}
Let  $g,h  \in |RL|^1_{\nu} (\mathbb{R})$
such that $g(s)\leq h(s),$ for every $\mathit{s\in S,}$
 then
\begin{description}
\item[(\ref{monotonicity}.a)]
			\quad
${\scriptstyle (|RL|)}\displaystyle{\int_S}g\, \mathrm{d}\nu \leq
{\scriptstyle (|RL|)}\displaystyle{\int_S}h\, \mathrm{d}\nu $.
\end{description}
Let
 $\nu_{1}$, $\nu _{2}:\mathcal{C}\rightarrow [0, +\infty )$
be  set functions such that $\nu _{1}(A)\leq \nu_{2}(A)$, for
every $A\in \mathcal{C}$
and $h\in |RL|^1_{\nu_i} (\mathbb{R}_0^+)$ for $i=1,2$ Then
\begin{description}
\item[(\ref{monotonicity}.b)]
		\quad ${\scriptstyle (|RL|)}\displaystyle{\int_S} h \, \mathrm{d}\nu _{1}\leq
		{\scriptstyle (|RL|)}\displaystyle{\int_S}h\, \mathrm{d}\nu _{2}$.
\end{description}
\end{theorem}

\begin{proof} We prove here (\ref{monotonicity}.a).
 Let $\varepsilon >0$ be arbitrary. Since $g, h\in |RL|^1_{\nu} (\mathbb{R})$, there exists
a countable partition
 $P_{0}$ so that for every
$P=\{C_{n}\}_{n\in \mathbb{N}}, P\geq P_{0}$ and
every $t_{n}\in C_{n},n\in \mathbb{N}$, the series
$\sum_{n=0}^{\infty}g(t_{n}) \nu (C_{n})$,   $\sum_{n=0}^{\infty}h(t_{n}) \nu (C_{n})$
are
absolutely convergent and
\begin{eqnarray*}
\max \left\{ \left| {\scriptstyle (|RL|)}\int_S g\, \mathrm{d}\nu -
\sum\limits_{n=0}^{\infty}g(t_{n})\nu (C_{n}) \right|,
\,
\left| {\scriptstyle (|RL|)}\int_S h\, \mathrm{d}\nu -\sum_{n=0}^{\infty}h(t_{n}) \nu (C_{n}) \right| \right\} <\frac{\varepsilon}{3}.
\end{eqnarray*}
Therefore
\begin{eqnarray*}
  {\scriptstyle (|RL|)}\int_S g \, \mathrm{d}\nu  -
{\scriptstyle (|RL|)}\int_S h \, \mathrm{d}\nu  =&&
{\scriptstyle (|RL|)} \int_S  g \, \mathrm{d}\nu  -
\sum_{n=0}^{\infty}g(t_{n})\nu  (C_{n}) +
\sum_{n=0}^{\infty}g(t_{n}) \nu  (C_{n}) + \\
\\-  &&
\sum_{n=0}^{\infty}h(t_{n})\nu (C_{n}) +
\sum_{n=0}^{\infty}h(t_{n}) \nu (C_{n})-
{\scriptstyle (|RL|)}\int_S h\, \mathrm{d}\nu
<  \\< &&
\dfrac{2\varepsilon }{3}+
\Big[\sum_{n=0}^{\infty}g(t_{n}) \nu (C_{n})-\sum_{n=0}^{\infty}h(t_{n}) \nu (C_{n}) \Big] \leq \varepsilon
\end{eqnarray*}
since, by the hypothesis,  $\sum_{n=0}^{\infty}g(t_{n})\nu(C_{n})\leq \sum_{n=0}^{\infty}h(t_{n})\nu (C_{n}).$\\
 Consequently,
\[{\scriptstyle (|RL|)}\int_S g \, \mathrm{d}\nu  -
{\scriptstyle (|RL|)}\int_S h \, \mathrm{d}\nu \leq 0.\]
\end{proof}

For every $h:S\rightarrow X$ that is $|RL|$  ($RL$ resp.)  $\nu$-integrable on every set $E\in \mathcal{C}$, we
consider the
 $|RL|$  integral operator
 $T_h :\mathcal{C}\rightarrow X$, defined for every $E\in \mathcal{C}$
by,
%\mg{If}
\begin{eqnarray}\label{funz-int}
T_{h}  (E)={\scriptstyle (|RL|)}\int_{E}h \, \mathrm{d}\nu \quad \big( T_{h}  (E)={\scriptstyle (RL)}\int_{E}h\, \mathrm{d}\nu\quad\textrm{resp.}\big)
\end{eqnarray}

We point out that, even without the additive condition for the set function $\nu$,
 the indefinite integral is additive thanks to Theorem \ref{3.1}.\\

In the next theorem we present some properties of the set
function $T_h$.

\begin{theorem}\label{properties}
\mbox{\rm (\cite[Theorem 8]{ccgis})}
Let $h \in |RL|^1_{\nu} (X)$. If  $h$ is bounded, and  $\nu$ is of finite variation then
\begin{description}
\item[(\ref{properties}.a)]
\begin{itemize}
	\item   $T_h$ is of finite variation too;
	\item  $\overline{T_{h}} \ll \overline{\nu} $ in the $\varepsilon - \delta$ sense;
	%(i.e. for every  $\varepsilon>0$, $\exists \delta>0$ such that
	%	$\forall E\in \mathcal{P}(S)$, $\overline{\nu}(E)<\delta$ $\Rightarrow$ $\overline{T_{h}}(E)<\varepsilon$);
	\item  Moreover, if
 $\overline{\nu}$ is o-continuous (exhaustive resp.), then $T_h$ is also o-continuous (exhaustive resp.).
\end{itemize}
\item[(\ref{properties}.b)]
  If $h:S \rightarrow [0,\infty)$ is nonnegative and $\nu $ is  scalar-valued and monotone, then the same holds for $T_h$.
\end{description}
\end{theorem}

\begin{proof} (\ref{properties}.a).
In order to prove that $T_h$ is of finite variation let $\{A_{i}\}_{i=1, \ldots, n}$ be a pairwise partition of $S$
 and $M=$ $\underset{s\in S}{\sup }\|h(s)\|.$ By Theorem \ref{upper-bound}.a)
% Corollary \ref{3.9}   %%%% qui ho cambiato richiamo, ar
 and Remark \ref{rem-var}, %-II, it
we have
\[ \sum_{i=1}^{n}\|T_h(A_{i})\|\leq
M\cdot\sum_{i=1}^{n}\overline{\nu}(A_{i}) \leq
M\cdot\overline{\nu}(S). \]
So $\overline{T_h}(S)\leq M$ $\overline{\nu}(S)$,
 that yields  $\overline{T_h}(S) < +\infty$. Now the absolute continuity  in the $\varepsilon - \delta$ sense follows from
Theorem \ref{upper-bound}.a).
\\
 Let $M$ as before. If $M=0$, then $h=0$, hence $T_h=0.$\\
If $M>0$, by
Theorem \ref{upper-bound}.a)
we have $\|T_h(A)\|\leq M\cdot \overline{\nu}(A),$
for every $A\in \mathcal{C}.$
So the o-continuity of $T_h$ follows from that of
 $\overline{\nu}$.
The proof of the exhaustivity is similar. \\
For  (\ref{properties}.b)
let  $A, B\in \mathcal{C}$ with $A\subseteq B$ and $\varepsilon >0$. Since $h$ is
$\nu$-integrable on $A$, there exists a countable partition $P_{1}=\{C_{n}\}_{n\in \mathbb{N}}$  of $A$
so that for every
 other finer  countable partition $P=\{A_{n}\}_{n\in \mathbb{N}},$ of $A$  and
every $t_{n}\in A_{n},n\in \mathbb{N}$, the series
 $\sum\limits_{n=0}^{\infty}h(t_{n})\nu (A_{n})$ is
absolutely convergent and
\begin{eqnarray}\label{11}
\left|T_h(A) -\sum_{n=0}^{\infty}h(t_{n})\nu(A_{n})\right|<\frac{\varepsilon}{2}.
\end{eqnarray}

Since $f$ is $\nu$-integrable on $B$, let
 $P_{2}=\{D_{n}\}_{n\in \mathbb{N}}$ a countable partition of $B$ with the same meaning as before for the set $A$.
\\
Let  $\widetilde{P}_{1}= \{C_{n}, B\setminus A\}_{n\in \mathbb{N}}$  and $\widetilde{P}_{1}\wedge P_{2}$
(both countable partitions of $B$).\\
  Let $P=\{E_{n}\}_{n\in\mathbb{N}}$ be
an arbitrary countable partition of $B$, with $P\geq \widetilde{P}_{1}\wedge P_{2}$.\\
We observe that $P^{''}=\{E_{n}\cap A\}_{n\in \mathbb{N}}$ is also a partition of $A$ and $P^{''}\geq P_{1}.$\\
If $t_{n}\in E_{n}\cap A, n\in \mathbb{N}$ we have
\[ \max \left\{
|T_h(B)  - \sum\limits_{n=0}^{\infty}h(t_{n})\nu(E_{n})|, \,
|T_h(A) -\sum\limits_{n=0}^{\infty} h(t_{n})\nu (E_{n}\cap A)| \right\} <\dfrac{\varepsilon}{2}.\]
 Therefore

\begin{eqnarray*}
T_h(A) - T_h(B) &\leq&
\left| T_h(A) -\sum\limits_{n=0}^{\infty}h(t_{n})\nu(E_{n}\cap A)\right|+
\\ &+&
\left[\sum\limits_{n=0}^{\infty}h(t_{n}) \nu (E_{n}\cap A)-\sum\limits_{n=0}^{\infty}h(t_{n}) \nu (E_{n}) \right]+\\
&+& \left|\sum\limits_{n=0}^{\infty}h(t_{n})\nu (E_{n})- T_h(B) \right|< \\
&<&
\varepsilon+ \left[\sum\limits_{n=0}^{\infty} h(t_{n})%
\nu (E_{n}\cap A)-\sum\limits_{n=0}^{\infty} h(t_{n})\nu (E_{n}) \right].
\end{eqnarray*}
Since, by the hypotheses,
$\sum\limits_{n=0}^{\infty}h(t_{n})\nu (E_{n}\cap A)\leq \sum\limits_{n=0}^{\infty}h (t_{n})\nu(E_{n})$,
then
 $T_h(A) \leq T_h(B) .$
  \end{proof}

\subsection{Comparison with other types of integrability}\label{sec-comp}
In the non additive case there are other types of integral that can be considered.
We present here some comparative results with the Gould and Birkhoff simple ones.
We recall that
\begin{definition}\label{3.10}
%\mg{3.10}
\rm (\cite[Definition 3.2]{CCGS})
A vector function $h:S\rightarrow X$ is called \textit{Birkhoff
simple $\nu$-integrable (on $S$)} if there exists $b\in X$ such that
for every $\varepsilon>0$, there exists a countable partition $P_{\varepsilon}$ of
$S$ so that for every other countable partition $P=\{A_{n}\}_{n\in
\mathbb{N}}$ of $S$, with $P\geq P_{\varepsilon}$ and every $s_{n}\in
A_{n}, n\in \mathbb{N},$ it holds
\[
\limsup_{n\rightarrow +\infty}
\Big\Vert \sum_{k=0}^{n}h(s_{k})\nu (A_{k})- b \Big\Vert <\varepsilon.
\]
The vector $b$ is denoted by $(Bs)\displaystyle{\int_S}h\, \mathrm{d}\nu$ and it is called \textit{the Birkhoff simple integral}
of $h$ (on $S$) with respect to $\nu$.
\end{definition}

Let
$(\mathcal{P}, \leq)$ be the family  of all finite partitions of $S$ ordered by the relation "$\leq$" given in \eqref{minore-ouguale}.
Given a vector function $h:S\rightarrow X$, we denote by $\sigma (P)$ the finite sum:
$\sigma(P) :=
 \sum\limits_{i=1}^{n}h(s_{i}) \nu (E_{i}),$ for every finite partition of $S$, $P=\{E_{i}\}_{i}^{n}\in \mathcal{P}$
 and every $s_{i}\in E_{i}, i\in\{1,\ldots, n\}.$ Following \cite{G}
\begin{definition}\label{3.11}
%\mg{3.11}
A function
$h:S\to X$ is called \textit{Gould $\nu$-integrable}(on $S$) if there exists
$a\in X$ such that for every $\varepsilon>0$, there exists a finite partition
$P_{\varepsilon}$ of $S$, so that for every other finite partition $P=\{E_{i}\}_{i=1}^{n}$
of $S$, with $P\geq P_{\varepsilon}$ and every $s_{i}\in E_{i}, i\in\{1, \ldots, n\},$ we have
$\|\sigma(P)-a\|<\varepsilon.$
The vector $a$ is called \textit{the Gould integral
of $h$ with respect to $\nu$}, denoted by $(G)\displaystyle{\int_S}h\, \mathrm{d}\nu.$
\end{definition}
Observe that
$h:S\rightarrow X$ is Gould $\nu$-integrable (on $S$) if and only if the
net $(\sigma(P))_{P\in (\mathcal{P},\leq)}$ is convergent in $X$,
The limit of
$(\sigma(P))_{P }$ is exactly the integral $(G)\int_{S}h\, \mathrm{d}\nu.$

Let  $Bs^1_{\nu} (S)$ and $G^1_{\nu}(S)$ be respectively the families of Birkhoff simple, Gould integrable functions.
In general $RL^1_{\nu}(S) \subset Bs^1_{\nu} (S)$ and the two integrals coincide, this is proved in \cite[Theorem 9]{ccgis},
while for what concernes the comparison between RL and Gould integrability for bounded functions we have the following relations:
\begin{itemize}
\item
 $RL^1_{\nu}(X)  = G^1_{\nu}(X)$
when
$\nu$ is  of finite variation and  defined on  a complete $\sigma$-additive measure by \cite[Proposition 2]{ccgis};
\item
 $RL^1_{\nu}(\mathbb{R})  = G^1_{\nu}(\mathbb{R})$
when $\nu$  is  of finite variation,  monotone and $\sigma$-subadditive
\cite[Theorem 10]{ccgis}.
\item  $RL^1_{\nu}(\mathbb{R})  \subset  G^1_{\nu}(\mathbb{R})$
on each atom $A \in \mathcal{C}$
when  $\nu$ is monotone, null additive and satisfies property $(\sigma)$, see \cite[Theorem 11]{ccgis}.
%In both cases  the two integrals coincide.
\end{itemize}
In all the  cases  the two integrals coincide.\\

Without  the $\sigma$-additivity of $\nu$, the second  equivalence  $RL^1_{\nu}(\mathbb{R})  = G^1_{\nu}(\mathbb{R})$
 for bounded functions does not hold. Suppose $S=\mathbb{N}$,
with $\mathcal{C}=\mathcal{P}(\mathbb{N})$ and
$$\nu(A)=\left\{\begin{array}{ll}
0, \ & card(A)<+\infty  \\
1, &  card(A)=+\infty,
\end{array}\right.
\quad \mbox{\rm for every }  A\in\mathcal{C}.
$$
Then, the constant function $h= 1$ is RL integrable  and then Birkhoff simple integrable and $(Bs)\int_{S}h\, \mathrm{d}\nu=0$.
However, $h$ is not Gould-integrable. In fact, if $P_{\varepsilon}$ is any finite partition of $\mathbb{N}$,
some of its sets are infinite, so the quantity $\sigma(P_{\varepsilon})$ is exactly  the number of the infinite sets belonging to $P_{\varepsilon}$.
 So the  quantity $\sigma(P)$ is unbounded when  $P$ runs over the family of all finer partitions of  $P_{\varepsilon}$.
\\

\subsection{Convergence results}
In
this subsection we want to quote some sufficient conditions in order to obtain, under suitable hypotheses,
a convergence result of this type
\[\lim\limits_{n\rightarrow \infty}{\scriptstyle (|RL|)}\int_{S}h_{n}\, \mathrm{d}\nu= {\scriptstyle (|RL|)}\int_{S}\lim\limits_{n\rightarrow \infty} h_n\, \mathrm{d}\nu,\]
 for  sequences of Riemann-Lebesgue integrable functions.
 We assume $\nu$ of finite variation  unless otherwise specified. Let $p\in [1, \infty)$ be fixed.
For every  real valued function $h:S\rightarrow \mathbb{R}$,
with $|h|^{p}\in RL^1_{\nu}(\mathbb{R})$, we associate the following number:
  \begin{eqnarray}\label{seminorm}
\Vert h\Vert_{p}= \Big( {\scriptstyle (|RL|)}\int_{S}|h|^{p} \, \mathrm{d}\nu \Big)^{\frac{1}{p}}.
\end{eqnarray}

\begin{theorem}
Let   $h,h_n : S \to X$, $\nu:\mathcal{C}\rightarrow [0,+\infty)$ and  $p\in [1, +\infty)$.
\begin{description}
\item[\mbox{\rm (\cite[Theorem 5]{cgis2022})}]
If $h, h_{n} \in |RL|^1_{\nu} (X)$  for every $n\in \mathbb{N}$ and
$h_{n}$  converges uniformly to $h$; or \\

\item[\mbox{\rm (\cite[Theorem 6]{cgis2022})}]
 If $X = \mathbb{R}$, $\sup_{s \in S, n \in \mathbb{N}} \big\{h(s), h_n(s) \big\} < +\infty$  and
$h_{n}\overset{\widetilde{\nu}}{\rightarrow} h$;
or\\

\item[\mbox{\rm (\cite[Theorem 8]{cgis2022})}]
If  $X = \mathbb{R}$,  $\nu$ is  monotone and
 $\widetilde{\nu}$ satisfies condition  \mbox{\rm \bf (E)}, $\sup_{s \in S, n \in \mathbb{N}} \big\{h(s), h_n(s) \big\} < +\infty$  and
    $h_{n}\overset{\nu-ae}{\longrightarrow} h$
\end{description}
then
\begin{eqnarray}\label{tesi}
\lim\limits_{n\rightarrow \infty}{\scriptstyle (|RL|)}\int_{S}h_{n}\, \mathrm{d}\nu= {\scriptstyle (|RL|)}\int_{S}h\, \mathrm{d}\nu.
\end{eqnarray}
Finally
\begin{description}
\item[\mbox{\rm (\cite[Theorem 7]{cgis2022})}]
If $X = \mathbb{R}$,  $\nu$ is countable subadditive (not necessarily of finite variation),
 $\chi_E \cdot |h_{n}-h|^{p}\in |RL|^1_{\nu} (\mathbb{R})$, for every  $E \in \mathcal{C}$ and $\|h_{n}-h\|_{p}\rightarrow 0.$
  Then $h_{n}\overset{\overline{\nu}}{\rightarrow} h.$
\end{description}
\end{theorem}
\begin{proof}
We give only the proof of \cite[Theorem 6]{cgis2022}.
According to \cite[Proposition 1]{ccgis} $g,\, g_{n},\,  g_{n}-g \in |RL|^1_{\nu}(\mathbb{R})$, for every $ n\in \mathbb{N}$.
 Let $\alpha\in (0,+\infty)$ such that:
\[ \sup_{s\in S, n\in \mathbb{N}} \big\{ |g(s)|,\, |g_{n}(s)-g(s)| \big\}<\alpha.\]
 Let $\varepsilon>0$ be fixed.
 By hypothesis, there is $n_{0}(\varepsilon)\in \mathbb{N}$ such that
for every $ n\geq n_{0}(\varepsilon)$
\[\widetilde{\nu}(\{s\in S;\,\,  |g_{n}(s)-g(s)|\geq \varepsilon/4\overline{\nu}(S)\}
<\varepsilon/4\alpha.\]
 Then, there exists $A_{n}\in \mathcal{C}$ such that
$\{s\in S; |g_{n}(s)-g(s)|\geq \varepsilon/4\overline{\nu}(S)\} \subset A_{n}$
 and $\widetilde{\nu}(A_{n})=\overline{\nu}(A_{n})<\varepsilon/4\alpha.$
Using \cite[Theorem 3 and Corollary 1]{ccgis},
 for every $n\geq n_{0}(\varepsilon)$, it holds that:
 \begin{eqnarray*}
 &&\left|{\scriptstyle (|RL|)}\int_{S}g_{n}d \nu - {\scriptstyle (|RL|)}\int_{S}g\, \mathrm{d}\nu \right|\leq
 \left|{\scriptstyle (|RL|)}\int_{A_{n}}(g_{n}-g) \, \mathrm{d}\nu \right|+
 \left|{\scriptstyle (|RL|)}\int_{A^c_{n}}(g_{n}-g)\, \mathrm{d}\nu \right|\leq\\
 &\leq&
\overline{\nu}(A_{n})\cdot \sup_{s\in A_{n}}|g_{n}(s)-g(s)|+
\overline{\nu}(A^{c}_{n})\cdot \sup_{s\in A^{c}_{n}}|g_{n}(s)-g(s)|
<\varepsilon,
 \end{eqnarray*}
and this yields the \eqref{tesi}.
\end{proof}

  The following theorem establishes a Fatou type result for sequences of Riemann-Lebesgue integrable functions.
  \begin{theorem}\label{quasif}
\mbox{\rm (\cite[Theorem 9]{cgis2022}  )}
    Suppose $\nu:\mathcal{C}\rightarrow [0,+\infty)$ is a monotone set function of finite variation  such that
  $\widetilde{\nu}$ satisfies \mbox{\rm \bf (E)}.
    For every $n\in \mathbb{N}$, let $h_{n}:S\rightarrow \mathbb{R}$ be such that $(h_{n})$ is uniformly bounded.
    Then
\begin{eqnarray}\label{fatou}
{\scriptstyle (|RL|)}\int_{S}(\liminf_n h_{n})\, \mathrm{d}\nu\leq
 \liminf_n \Big({\scriptstyle (|RL|)}\int_{S}h_{n}\, \mathrm{d}\nu \Big).
\end{eqnarray}
  \end{theorem}
And its consequence
  \begin{corollary}\label{3.24}
\mbox{\rm (\cite[Theorem 10]{cgis2022}  )}
    Suppose $p\in (1, +\infty)$ and $\nu:\mathcal{C}\rightarrow [0,+\infty)$ is a monotone set function of finite variation
    such that $\widetilde{\nu}$ satisfies \mbox{ \rm (E)}.
    Let $h, \, h_n :S\rightarrow \mathbb{R}$ be  such that $h$ is bounded and
     $(h_{n})_n$ is pointwise convergent to $h$.
Let
$$g_{n}= 2^{p-1}(|h_{n}|^{p}+|h|^{p})-|h_{n}-h|^{p},$$  such that
    $(g_{n})$ is uniformly bounded, $|h|^{p}, |h_{n}|^{p}, |h_{n}-g|^{p}, g_{n},
    \inf_{k\geq n}g_{k}\in|RL|^1_{\nu}(\mathbb{R})$, for every $n\in \mathbb{N}$
    and $\|h_{n}\|_{p}\longrightarrow \|h\|_{p}.$ Then
$$\|h_{n}-h\|_{p}\longrightarrow 0.$$
    \end{corollary}
%================ subsection
\subsection{ H\"{o}lder and Minkowski type inequalities}
    In the end of this section, we expose a result on the reverse inequalities of
 H\"{o}lder and Minkowski type in Riemann-Lebesgue integrability. First of all we need that $\nu$ satisfies the following property
\begin{definition} \label{def-set}
  The set function $\nu:\mathcal{C}\rightarrow [0,\infty)$ is called RL-integrable
  if for all $E\in \mathcal{C}, \chi_{E}\in RL^1_{\nu}(\mathbb{R})$ and $\int_{S}\chi_{E}\, \mathrm{d}\nu= \nu(E).$
\end{definition}
%We recall from \cite{cgis2022} the following theorem, which shows
%H\"older's and Minkowski's inequalities for $p\geq 1.$
\begin{theorem}\label{questo}
 \mbox{\rm (\cite[Theorem 4]{cgis2022} and \cite[Theorem 3.4]{{cgis2023}})}
  Let $\nu:\mathcal{C}\rightarrow [0,\infty)$ be a countable subadditive RL-integrable
  set function and let $g, h:S\rightarrow \mathbb{R}$ be measurable functions.\\
Let $p, q\in (1, \infty)$, with  $p^{-1}+q^{-1}= 1$.\,
\begin{description}
\item[(\ref{questo}.a)]   If \, $g\cdot h
\in RL^1_{\nu}(\mathbb{R})$,  then
\begin{equation*}
\|g\cdot h\|_{1}\leq \|g\|_{p}\cdot\|h\|_{q} \quad \text{\rm (H\"{o}lder Inequality)}.
\end{equation*}
\item[(\ref{questo}.b)] Let $p\in [1, \infty)$. If $|g+h|^{p}, |g+h|^{q(p-1)}, |g|^{p}$ and
$|h|^{p}$ are in $\in RL^1_{\nu}(\mathbb{R})$, then
\begin{equation*}
  \|g+h\|_{p}\leq \|g\|_{p}+ \|h\|_{p} \quad \text{\rm (Minkowski  Inequality).}
\end{equation*}
\end{description}
 Let $p, q \in (0, \infty)$ such that $0<p<1$ and $p^{-1}+q^{-1}= 1. $
\begin{description}
\item[(\ref{questo}.c)]  If $g \cdot h, |g|^{p}, |h|^{q}
\in RL^1_{\nu} (\mathbb{R})$ and $0<{\scriptstyle (RL)}\int_{S}|h|^{q}\, \mathrm{d}\nu$,  then
\begin{equation*}
\|g\cdot h\|_{1}\geq \|g\|_{p}\cdot\|h\|_{q} \quad \text{\rm  (Reverse H\"{o}lder Inequality)}.
\end{equation*}
\item[(\ref{questo}.d)] If $|g+h|^{p}, |g+h|)^{(p-1)q}, |g|^{p}$ and $|h|^{p}$
are $RL$-integrable, then
\begin{equation*}
  \big\Vert \, |g|+|h| \, \big\Vert_{p}\geq ||g||_{p}+ ||h||_{p} \quad \text{ \rm (Reverse Minkowski  Inequality).}
\end{equation*}
\end{description}
\end{theorem}

According to \cite[Remark 4]{cgis2022}, for $p\in [1, \infty)$, the function $\| \cdot\|_{p}$ defined in \eqref{seminorm} is a seminorm on the linear space
of measurable $RL$-integrable functions.

\begin{proof} We give here only the proof of the reverse part.
\begin{description}
\item[(\ref{questo}.c)] If ${\scriptstyle (RL)}\int_{S}g^{p}\, \mathrm{d}\nu =0$, then according to \cite[Theorem 3]{cgis2022}
  it follows $g \cdot h=0$ $\nu-a.e.$ In this case, the inequality of integrals is satisfied.\\
Consider ${\scriptstyle (RL)}\int_{S}g^{p}\, \mathrm{d}\nu >0$.
   We replace $a= \frac{|g|}{({\scriptstyle (RL)}\int_{S}|g|^{p}\, \mathrm{d}\nu)^{\frac{1}{p}}}$ and $b= \frac{|h|}{({\scriptstyle (RL)}\int_{S}|h|^{q}\, \mathrm{d}\nu)^{\frac{1}{q}}}$ in
  the reverse Young inequality $ab\geq \dfrac{a^{p}}{p}+\dfrac{b^{q}}{q}$, for every $a, b>0$ and for every $0<p<1$ with
  $\frac{1}{p}+ \frac{1}{q}=1$ (see for example \cite{Adams,chen}). Then
\begin{eqnarray*}
  \frac{|gh|}{( {\scriptstyle (RL)}\int_{S}|g|^{p}\, \mathrm{d}\nu)^{\frac{1}{p}}({\scriptstyle (RL)}\int_{S}|h|^{q}\, \mathrm{d}\nu)^{\frac{1}{q}}}
  \geq \frac{|g|^{p}}{p({\scriptstyle (RL)}\int_{S}|g|^{p}\, \mathrm{d}\nu)}+ \frac{|h|^{q}}{q({\scriptstyle (RL)}\int_{S}|h|^{q}\, \mathrm{d}\nu)}.
\end{eqnarray*}
  Applying \cite[Theorems 3 and 6 ]{ccgis} it holds
\begin{eqnarray*}
  \frac{{\scriptstyle (RL)}\int_{S}|gh|\, \mathrm{d}\nu}{({\scriptstyle (RL)}\int_{S}|g|^{p}\, \mathrm{d}\nu)^{\frac{1}{p}}({\scriptstyle (RL)}\int_{S}|h|^{q}\, \mathrm{d}\nu)^{\frac{1}{q}}}
  \geq \frac{{\scriptstyle (RL)}\int_{S}|g|^{p}\, \mathrm{d}\nu}{p(\int_{S}|g|^{p}\, \mathrm{d}\nu)}+
\frac{{\scriptstyle (RL)}\int_{S}|h|^{q}\, \mathrm{d}\nu}{q({\scriptstyle (RL)}\int_{S}|h|^{q}\, \mathrm{d}\nu)}
=1
\end{eqnarray*}
and the conclusion yields.
\item[(\ref{questo}.d)]
 By (\ref{questo}.c), it results:
\begin{eqnarray*}
 {\scriptstyle (RL)} \int_{S}(|g|+|h|)^{p}\, \mathrm{d}\nu &=&
   {\scriptstyle (RL)}\int_{S}(|g|+|h|)^{p-1}(|g|+|h|)\, \mathrm{d}\nu\geq \\
&\geq& ( {\scriptstyle (RL)}\int_{S}(|g|+|h|)^{q(p-1)}\, \mathrm{d}\nu)^{\frac{1}{q}} ( {\scriptstyle (RL)}\int_{S}|g|^{p}\, \mathrm{d}\nu)^{\frac{1}{p}}+\\
  &+& ( {\scriptstyle (RL)}\int_{S}(|g|+|h|)^{q(p-1)}\, \mathrm{d}\nu)^{\frac{1}{q}}(\int_{S}|h|^{p}\, \mathrm{d}\nu)^{\frac{1}{p}}=\\
  &=&( {\scriptstyle (RL)}\int_{S}(|g|+|h|)^{q(p-1)}\, \mathrm{d}\nu)^{\frac{1}{q}}(\|g\|_{p}+ \|h\|_{p}).
\end{eqnarray*}
Dividing the above inequality by $( {\scriptstyle (RL)}\int_{S}(|g|+|h|)^{q(p-1)}\, \mathrm{d}\nu)^{\frac{1}{q}}$,
we obtain the Reverse Minkowski inequality.
\end{description}
\end{proof}
%===============================================
\section{Interval-valued Riemann-Lebesgue integral}
In the last years, a particular attention was addressed to the study of interval-valued multifunctions
and multimeasures because of their applications in statistics, biology, theory of games, economics, social sciences
and  software. Interval-valued multifunctions have been applied also to some new  directions,
 involving signal and image processing.\\
Motivated by the large number of fields in which the interval-valued multifunctions can be involved,
we present some classic properties for the Riemann-Lebesgue integral of an interval-valued multifunction
 with respect to an interval-valued set multifunction.

We begin by recalling some preliminaries.
The symbol $ck(\mathbb{R})$  denotes   the family of all non-empty,  convex,  compact subsets of $\mathbb{R}$,
  by convention, $\{0\}=[0,0]$. \\

   We consider on $ck(\mathbb{R})$ (\cite{hp}) the   Minkowski addition
   \[A\oplus B :=\{a+b\,\, |\,\, a\in A, \,\, b\in B\},
\quad \mbox{ for every } A, B\in ck(\mathbb{R})
\]
and the  multiplication by scalars
 \[\lambda A=\{\lambda a\,\, |\,\,  a\in A\},
\quad \mbox{ for every } \lambda \in \mathbb{R}, A\in ck(\mathbb{R}).
\]
   $d_H$ denotes the Hausdorff distance in $ck(\mathbb{R})$
and it is defined for every $ A, B\in ck(\mathbb{R}),$ in the following way:
   \[ d_H(A,B)=\max\{e(A,B), \, e(B,A)\},\]
 where $e(A,B) = \sup \{ d(x,B),\,\,  x \in A \}$. We use the symbol  $\|A\|_{H}$ to denote $d_H(A, \{0\})$.
In particular,  for closed intervals we have
   \begin{eqnarray*}
&& d_H([r,s], [x,y]) = \max \{|x-r|,|y-s|\}, \quad \mbox{\rm for every } r,s,x,y \in \mathbb{R};
\\
&& d_H([0,s], [0,y]) = |y-s|, \quad \mbox{\rm for every } s,y \in \mathbb{R}_{0}^{+};
\\
&& \|[r,s]\|_{\mathcal{H}}= s, \quad \mbox{\rm for every } r, s\in \mathbb{R}_{0}^{+}.
\end{eqnarray*}

In the subfamily $L(\mathbb{R})$  of   intervals in $ck(\mathbb{R})$  the following operations are also considered, for every $r,s,x,y \in \mathbb{R}$
(see the Interval Analysis in  \cite{moore66}):
\begin{description}
  \item[i)] $ [r,s] \smallbullet [x, y] = [rx, sy]$;
  \item[ii)] $[r,s] \subseteq [x,y]$ if and only if $x \leq r \leq s \leq y$;
  \item[iii)] $[r,s] \preceq [x,y]$ if and only if $r \leq x$ and $s \leq y$;  (weak interval order, \cite{GZ})
  \item[iv)] $[r, s] \wedge  [x, y] = [\min\{r,x\}, \min \{s,y\}]$;
  \item[v)] $[r, s] \vee  [x, y] = [\max\{r,x\}, \max \{s,y\}]$.
\end{description}
In general there is no relation between the weak interval order and the inclusion; they only coincide
on the subfamily  of intervals $[0,s], s\geq 0.$\\
For every pair of sequences of real numbers  $(u_{n})_{n}, (v_{n})_{n}$ such that $0\leq u_{n}\leq v_{n}$,
 for every $ n\in\mathbb{N}$, we define:
\begin{description}
\item[vi)] $\inf\limits_{n}[u_{n},\,  v_{n}]= [\inf\limits_{n} u_{n}, \,  \inf\limits_{n} v_{n}];$
\item[vii)] $\sup\limits_{n}[u_{n},\,  v_{n}]= [\sup\limits_{n} u_{n}, \, \sup\limits_{n} v_{n}];$
\item[viii)] $\liminf\limits_{n}[u_{n},\,  v_{n}]= [\liminf\limits_{n} u_{n},\, \liminf\limits_{n} v_{n}].$
\end{description}

We consider $(ck(\mathbb{R}_{0}^{+}), d_H, \preceq)$, namely the space $ck(\mathbb{R}_0^+)$
is endowed with the Hausdorff distance and the weak interval order.\\

For two set functions $\nu_{1}, \nu_{2} : \mathcal{C} \rightarrow \mathbb{R}_{0}^{+}$ with
$\nu_1(\emptyset)=\nu_2(\emptyset)=0$ and
$\nu_{1}(A) \leq \nu_{2}(A)$
 for every $A \in \mathcal{C}$ the set multifunction $\Gamma: \mathcal{C} \rightarrow L(\mathbb{R}_{0}^{+})$ defined by

\begin{eqnarray}\label{M}
\Gamma(A) = \big[ \nu_{1}(A), \nu_{2}(A) \big], \qquad \mbox{ for every  } A\in \mathcal{C}.
\end{eqnarray}
is called an interval-valued set function.
 In this case, $\Gamma$ is of finite variation if and only if $\nu_{2}$ is of finite  variation.\\

Let $\Gamma: \mathcal{C} \rightarrow L(\mathbb{R}_{0}^{+})$. We say that $\Gamma$ is an interval-valued multisubmeasure if
\begin{itemize}
\item \quad  $\Gamma (\emptyset )=\{0\};$
\item \quad  $\Gamma(A) \preceq \Gamma(B)$ for every $A, B \in \mathcal{C}$ with $A \subseteq B$  (monotonicity);
\item \quad $\Gamma(A \cup B) \preceq \Gamma(A) \oplus \Gamma(B)$ for every disjoint sets $A, B \in \mathcal{C}$.
(subadditivity).
\end{itemize}
By \cite[Remark 3.6]{pig} $\Gamma$ is a multisubmeasure with respect to $\preceq$
 if and only if $\nu_i$, $i=1,2$ are submeasures.
\begin{definition} \label{D_mult}
It is said that $\Gamma$ is a $d_H$-multimeasure if for every sequence of pairwice disjoint sets
$(A_n)_n \subset \mathcal{C}$ such that $\cup_{n}^{\infty}A_{n} = A$
\[
\lim_{n \to \infty} d_H(\sum_{k=1}^{n} \Gamma(A_{k}), \Gamma(A)) = 0,\]
\end{definition}
\subsection{The interval-valued RL-integral and its properties}
In what follows, all the interval-valued set functions we consider are   multisubmeasures.

Given $h_{1}, h_{2}: S \rightarrow \mathbb{R}_{0}^{+}$ with $h_{1}(s) \leq h_{2}(s)$ for all $s \in S$,
let $H: S \rightarrow L(\mathbb{R}_{0}^{+})$ be the interval-valued multifunction
 defined by
 \begin{equation} \label{H}
 H(s) := \big[h_{1}(s), h_{2}(s) \big], \qquad \mbox{for every  }  s \in S.
 \end{equation}
 $H$ is bounded if and only if $h_{2}$ is bounded.
 If $H, G:S\to L(\mathbb{R}_{0}^{+})$ are as in (\ref{H}) so that $G\preceq H$ or $G\subset H$ and
 $H$ is bounded, then $G$ is bounded too.
For every countable tagged partition $P =\{ (B_{n}, s_{n}),  n \in \mathbb{N}\}$ of $S$ we denote by
\begin{eqnarray*}
\sigma_{H,\Gamma} (P) &=& \sum_{n=1}^{\infty} H(s_{n}) \smallbullet \Gamma(B_{n})
=  \sum_{n=1}^{\infty} \big[h_{1} (s_{n}) \nu_{1} (B_{n}), h_{2}(s_{n}) \nu_{2} (B_{n}) \big] = \\
&=& \Big\{ \sum_{n=1}^{\infty} y_{n}, y_{n} \in  \big[h_{1} (s_{n}) \nu_{1}(B_{n}), h_{2}(s_{n}) \nu_{2} (B_{n}) \big], n \in \mathbb{N} \Big\}.
  \quad
\end{eqnarray*}
The set $\sigma_{H, \Gamma} (P)$ is closed and convex in $\mathbb{R}_{0}^{+}$, so it is an interval
$\big[h_{1,H, \Gamma}^{P}, h_{2,H, \Gamma}^{P} \big]$.\\
\begin{definition}
\rm
A  multifunction $H:S\rightarrow L(\mathbb{R}_{0}^{+})$ is called
 Riemann-Lebesgue (RL in short) integrable with respect to $\Gamma$ (on $S$) if there exists $[c,d] \in L(\mathbb{R}_{0}^{+}) $
such that for every $\varepsilon>0$, there exists a countable partition
$P_{\varepsilon}$ of $S$, so that for every tagged partition
$P=\{(B_{n}, s_{n})\}_{n\in \mathbb{N}}$ of $S$ with $P\geq P_{\varepsilon}$,
\begin{itemize}
\item \quad the series $\sigma_{H, \Gamma} (P)$ is convergent with respect to the Hausdorff distance $d_H$ and
\item \quad
$d_H (\sigma_{H, \Gamma} (P), [c,d]) <\varepsilon$.
\end{itemize}
The interval $[c,d]$ is called the Riemann-Lebesgue integral of $H$ with respect to $\Gamma$ and it is denoted
$$[c,d]= {\scriptstyle (RL)}\int_{S} H \, \, \mathrm{d}\Gamma.$$
The symbol $RL^1_{\Gamma}(L(\mathbb{R}_{0}^{+}))$ denotes the class of all interval-valued functions that are Riemann-Lebesgue  integrable with respect to $\Gamma$ on $S$.
\end{definition}

\begin{example} \rm (\cite[Example 5]{cgis2022})
 Suppose $S = \{s_{n}\, | \, n\in \mathbb{N}\}$ is countable, $\{s_{n}\}\in \mathcal{C}$, for every $n \in \mathbb{N}$, and let
$H : S \rightarrow L(\mathbb{R}_{0}^{+})$   be such that the series
$\displaystyle\sum_{n=0}^{\infty}h_{i}(s_{n})\nu_{i}(\{s_{n}\}), \, i\in\{1,2\}$, are convergent.
Then $H$ is  $RL$ integrable  with respect to $\Gamma$ and
\[
{\scriptstyle (RL)}\int_{S} H\, \, \mathrm{d}\Gamma=
\left[\sum_{n=0}^{\infty}h_{1}(s_{n})\nu_{1}(\{s_{n}\}),
\sum_{n=0}^{\infty}h_{2}(s_{n})\nu_{2}(\{s_{n}\})\right].
\]
Observe moreover that, in this case, the  $RL$-integrability of $H$  with respect to $\Gamma$ implies that the product $H \smallbullet H$
is integrable in the same sense, where
$(H {\scalebox{1.8}{$\cdot$}} H)(s)=[h_{1}^{2}(s), h_{2}^{2}(s)]$, for every $s\in S$.
 In particular, if $H$ is a discrete or countable interval-valued signal, then the integral
${\scriptstyle (RL)}\int_{S} H \smallbullet H\, \, \mathrm{d}\Gamma$ represents the energy of the signal, see for example \cite[Example 2]{danilo}.\\
\end{example}
If $\Gamma$ is of finite  variation and $H: S\rightarrow L(\mathbb{R}_{0}^{+})$ is bounded and
such that $H=\{0\}$\, $\Gamma$-a.e., then $H$ is $\Gamma$-integrable and
${\scriptstyle (RL)}\int_{S}H\, \mathrm{d}\Gamma= \{0\}.$\\

In the sequel, some properties of interval-valued $RL$ integrable multifunctions (\cite{danilo}) are presented for $\Gamma$
as in (\ref{M}) and the multifunctions as in (\ref{H}).
The following theorem shows a characterization of the $RL$ integrability.
\begin{theorem}\label{4,3}
\mbox{\rm (\cite[Proposition 2]{danilo})}
	An interval-valued multifunction $H=[h_{1}, h_{2}]$
%as in (\ref{H})
 is $RL$ integrable with respect to $\Gamma$ on $S$
if and only if $h_{1}$ and $h_{2}$ are $RL$ integrable with respect to $\nu_{1}$ and $\nu_{2}$ respectively and

	\begin{eqnarray*}
{\scriptstyle (RL)} \int_{S} H \, \mathrm{d}\Gamma =
 \Big[ {\scriptstyle (RL)}\int_{S} h_{1} \, \mathrm{d}\nu_{1}, {\scriptstyle (RL)}\int_{S} h_{2}  \, \mathrm{d}\nu_{2} \Big].
	\end{eqnarray*}
\end{theorem}
\begin{proof}
By the  $RL$ integrability of $H$  there exists $[a,b]$
such that for every $\varepsilon>0$, there exists a countable partition
$P_{\varepsilon}$ of $S$, so that for every tagged partition
$P=\{(A_{n}, t_n)\}_{n\in \mathbb{N}}$ of $S$ with $P\geq P_{\varepsilon}$,
the series $\sigma_{H,\Gamma} (P)$ is
convergent and
\begin{eqnarray*}
&& \max\{ | \sum_{n=1}^{\infty} h_1 (t_n) \nu_1 (A_n) - a| ,\,
 | \sum_{n=1}^{\infty} h_2 (t_n) \nu_2 (A_n) - b| \,\}=
\\&&
d_H (\sum_{n=1}^{\infty} \big[h_{1} (s_{n}) \nu_{1} (B_{n}), h_{2}(s_{n}) \nu_{2} (B_{n}) \big]
 <\varepsilon.
\end{eqnarray*}
for every tagged partition
$P=\{(A_{n}, t_n)\}_{n\in \mathbb{N}}$ of $S$ with $P\geq P_{\varepsilon}$ and then
 $h_i$ are $RL$ integrable with respect to $\nu_i$, $i=1,2$. So the first implication  follows from the convexity of the $RL$ integral.\\
  For the converse, for every $\varepsilon>0$, let $ P_{\varepsilon, h_i}, i=1,2$ be two countable partitions verifying  the  $RL$ integrability definition  for $h_i, i=1,2$ respectively.
Let $P_{\varepsilon} \geq
P_{\varepsilon, h_1} \wedge P_{\varepsilon, h_2}$ be  a countable partition of $S$,
%\{ P_{\varepsilon, g_i}, \,\, i=1,2  \}$.
 then, for every  finer partition
$P:=( B_n)_n$ and for every $t_n \in B_n$ it is
\begin{eqnarray*}
\left|\sum_{n=0}^{+\infty}h_i(t_{n}) \nu_i (B_{n})- {\scriptstyle (RL)}\int_S h_i d\nu_i \right|<\varepsilon, \quad i=1,2.
\end{eqnarray*}
Since $h_i, \, i=1,2$ are selections of $H$ then
$$d_H \left(\sigma_{H,\Gamma} (P),
 \left[ {\scriptstyle (RL)}\int_S h_1 d\nu_1, \,
 {\scriptstyle (RL)}\int_S h_2  d\nu_2\right] \right) \leq \varepsilon$$
and then the only if assertion follows.
\end{proof}

 According to Theorem \ref{upper-bound} and the previous theorem, if $\nu_{2}$ is of finite variation and $h_{2}$ is bounded, then $H$ is RL integrable. Another consequence of the previous theorem, together with Theorem \ref{3.1}.b)
is the inheritance of the $RL$ integral on the subsets  $E\in\mathcal{C}$. In fact
\begin{corollary} \label{4.4}
Let $H \in RL^1_{\Gamma}(L(\mathbb{R}_{0}^{+}))$, then $H$ is
 $RL$ integrable  with respect to  $\Gamma$
 on every $E\in\mathcal{C}$.
Moreover, $H$ is  $RL$ integrable  with respect to  $\Gamma$ on $E\in \mathcal{C}$ if and only if $H\chi _{E}$
  is  $RL$ integrable  with respect to  $\Gamma$ on $S$.
 In this case, for every $E \in \mathcal{C}$,
  $${\scriptstyle (RL)}\int_{E} H \, \, \mathrm{d}\Gamma =
{\scriptstyle (RL)}\int_{S} H\chi _{E}\, \, \mathrm{d}\Gamma .$$
\end{corollary}

Moreover the RL integral is homogeneous with respect to both interval-valued multifunctions $H$ and $\Gamma$.
\begin{theorem} \label{4.5}
\mbox{\rm (\cite[Remark 7, Theorem 1 and Proposition 8]{danilo})}\\
If  $H, H_{1}, H_{2} \in RL^1_{\Gamma}(L(\mathbb{R}_{0}^{+}))$
then for every $\alpha\in [0, \infty)$:
\begin{description}
\item[(\ref{4.5}.a)] $\alpha H \in RL^1_{\Gamma}(L(\mathbb{R}_{0}^{+}))$  and
\[{\scriptstyle (RL)}\int_{S}\alpha H \, \mathrm{d}\Gamma =
\alpha {\scriptstyle (RL)} \int_{S} H  \, \mathrm{d}\Gamma.\]

\item[(\ref{4.5}.b)] $H \in RL^1_{\alpha \Gamma}(L(\mathbb{R}_{0}^{+}))$   and
\[{\scriptstyle (RL)}\int_{S} H \, \mathrm{d}(\alpha \Gamma)=
 \alpha  \int_{S} H \, \mathrm{d}\Gamma. \]

\item[(\ref{4.5}.c)] $H_{1} \oplus H_{2} \in RL^1_{\Gamma}(L(\mathbb{R}_{0}^{+}))$  and
\[{\scriptstyle (RL)}\int_{S}(H_{1}\oplus  H_{2})\, \mathrm{d}\Gamma =
{\scriptstyle (RL)}\int_{S} H_{1} \, \mathrm{d}\Gamma \oplus
{\scriptstyle (RL)}\int_{S} H_{2} \, \mathrm{d}\Gamma.\]
\end{description}
\end{theorem}

\begin{proof} We give here the proof of (\ref{4.5}.c).
Namely we prove that for every pair of interval-valued multifunctions $H_1,H_2$, which are $RL$ integrable  with respect to $\Gamma$  we have that
%\mg{add}
\begin{eqnarray}\label{add} {\scriptstyle (RL)} \int_S (H_1 \oplus H_2) \, \mathrm{d}\Gamma =
{\scriptstyle (RL)} \int_S H_1 \, \mathrm{d}\Gamma \oplus
{\scriptstyle (RL)}\int_S  H_2\, \mathrm{d}\Gamma.
\end{eqnarray}
Let $\varepsilon >0$ be fixed. Since $H_1,H_2$ are $RL$ integrable with respect to $\Gamma$,
  there exists a countable partition
 $P_{\varepsilon}\in \mathcal{P}$ such that for every $P=\{A_{n}\}_{n\in \mathbb{N}} \geq P_{\varepsilon}$ and every
 $t_{n}\in A_{n},\, n\in \mathbb{N}$, the series $\sigma_{H_i,\Gamma}(P)$, $i=1,2$
 are  convergent and
\begin{eqnarray*}
&&
d_H\left(\sigma_{H_i,\Gamma}(P),{\scriptstyle (RL)}\int_S H_i \, \mathrm{d}\Gamma  \right) <\dfrac{\varepsilon }{2}, \qquad i=1,2.
\end{eqnarray*}
Then $\sigma_{H_1\oplus H_2,\Gamma}(P)$ is convergent and, by \cite[Proposition 1.17]{hp},
\begin{eqnarray*}
d_H\left(\sigma_{H_1 \oplus H_2,\Gamma}(P),{\scriptstyle (RL)}\int_S H_1 \, \mathrm{d}\Gamma \oplus
 {\scriptstyle (RL)}\int_S H_2 \, \mathrm{d}\Gamma   \right) <\varepsilon.
\end{eqnarray*}
So  $H_1\oplus H_2$ is $RL$ integrable with respect to $\Gamma$  and formula (\ref{add}) is satisfied.
%\\ Now applying formula (\ref{add})
% with $G_1 = H \chi_A, \, G_2 =H \chi_B$  to formula (\ref{chi}) we obtain the additivity of $T_H$.
\end{proof}

If $H \in RL^1_{\Gamma}(L(\mathbb{R}_{0}^{+}))$, then we may consider
$T_{H}: \mathcal{C} \rightarrow L(\mathbb{R}_{0}^{+})$ defined by
\begin{eqnarray}\label{TH}
T_{H} (E) = {\scriptstyle (RL)}\int_{E} H \, \mathrm{d}\Gamma,
 \qquad \mbox{for every  }  E\in \mathcal{C}.
\end{eqnarray}

In the following theorem we present some properties of the interval-valued integral set operator $T_{H}$.
\begin{theorem}  \label{op}
 Let $\Gamma:\mathcal{C} \rightarrow L(\mathbb{R}_{0}^{+})$ be  so that $\nu_{2}$ is of finite variation
 and $H:S\to L(\mathbb{R}_{0}^{+})$ is bounded. Then the following properties hold:
\begin{description}
\item[(\ref{op}.a)] $T_{H}$ is a finitely additive multimeasure, i.e. for every $A, B\in \mathcal{C}$, with $A\cap B=\emptyset$ it is
$T_{H}(A\cup B)= T_{H}(A)\oplus T_{H}(B).$ \mbox{\rm (\cite[Theorem 1]{danilo})}.
\item[(\ref{op}.b)] Let  $G, H \in RL^1_{\Gamma}(L(\mathbb{R}_{0}^{+}))$.
Then,  for every $E\in \mathcal{C}$, by \mbox{\rm \cite[Propositions 4 and 5]{danilo}},
if $G \preceq H$, then $T_{G} (E) \preceq T_{H}(E)$; \,\,
 if  $G\subseteq H$,  then $T_{G}(E) \subseteq T_{H}(E)$.\\
% If moreover  $\Gamma$ is a multisubmeasure,
Moreover, by \mbox{\rm \cite[Corollary 1]{danilo}}, for every $E \in \mathcal{C}$:
\[ T_{G \wedge H}(E)  \preceq
 T_{G} (E) \wedge T_{H}(E); \quad
 T_{G} (E) \vee T_{H}(E) \preceq T_{G \vee H}(E)
\]
\end{description}
Finally from \mbox{\rm \cite[Propositions 6 and 7, Theorem 2]{danilo}} we have that
\begin{description}
\item[(\ref{op}.c)]
\begin{itemize}
\item
$\|T_{H}(S) \|_{\mathcal{H}}=
{\scriptstyle (RL)}\int_{S} h_{2} \, \mathrm{d}\nu_{2}=
{\scriptstyle (RL)}\int_{S}\|H\|_{\mathcal{H}}\, \mathrm{d}\|\Gamma\|_{\mathcal{H}}.$
\item
$\overline{T}_{H}(S) = {\scriptstyle (RL)}\int_{S} h_{2}\, \mathrm{d}\nu_{2}.$
\item
$T_H \ll \overline{\Gamma} $ (in the $\varepsilon$ - $\delta$ sense) and
$T_H$ is of finite variation.
\item
 If moreover $\Gamma$ is o-continuous (exhaustive resp.), then $T_{H}$ is also o-continuous
 (exhaustive resp.).
\item If $\Gamma $ is monotone, then  $T_H$ is monotone too.

\item If  $\Gamma$ is  a $d_H$-multimeasure, then $T_{H}$ is countably additive.
\end{itemize}
\end{description}
 \end{theorem}
\begin{proof}
We  point out that the additivity of $T_H$ is indipendent of the additivity of $H$. In fact, by Corollary \ref{4.4} we have that $T_H(A) \in L(\mathbb{R}_0^+)$ for every $A \in \mathcal{C}$.
 Moreover for every $A, B \in \mathcal{C}$ with $A \cap B = \emptyset$,
by  Theorem \ref{4.5}.c)
%\mg{chi}
\begin{eqnarray*}\label{chi}
T_H (A \cup B) &=& \displaystyle{{\scriptstyle (RL)}\int_S} H \chi_{A \cup B}\, \mathrm{d}\Gamma =
\displaystyle{{\scriptstyle (RL)}\int_S} (H \chi_{A} \oplus H \chi_{ B})\, \mathrm{d}\Gamma = \\
 &=& \displaystyle{{\scriptstyle (RL)}\int_S} H \chi_{A}  \mathrm{d}\Gamma
\oplus
\displaystyle{{\scriptstyle (RL)}\int_S} H \chi_{ B}\, \mathrm{d}\Gamma
= T_H (A) \oplus T_H (B).
\end{eqnarray*}

\end{proof}

\begin{comment}
 \begin{remark}
\rm
 An important property is that the integral set operator $T_{H}$ is finite additive even $H$ and $\Gamma$
 do not satisfy this property of finite additivity.
 \end{remark}
\end{comment}
 The RL integral is additive and monotone with respect to the weak interval order and the inclusion one relative to
 $\Gamma$, as we can see in the following theorem.
\begin{theorem}   \label{T9}
\mbox{\rm (\cite[Theorems 3 and  4]{danilo})}
Let $\Gamma _{1},\,  \Gamma _{2}:\mathcal{A}\rightarrow L(\mathbb{R}_{0}^{+})$ be multisubmeasures of finite variation,
with $\Gamma _{1}(\emptyset )= \Gamma_{2}(\emptyset )=\{0\}$ and suppose
$H,G:S\rightarrow L(\mathbb{R}_{0}^{+})$ are bounded multifunctions. Then
 the following properties hold for every $E \in \mathcal{C}$:
\begin{description}

\item[(\ref{T9}.a)]
If \, $\Gamma := \Gamma_1 \oplus \Gamma_2$,
% and  $H \in RL^1_{\Gamma_i}(L(\mathbb{R}_{0}^{+}))$ for $i=1,2$,
then
%$H \in RL^1_{\Gamma}(L(\mathbb{R}_{0}^{+}))$ and
\quad
${\scriptstyle (RL)}\int_{E} H\, \mathrm{d}\Gamma= {\scriptstyle (RL)}\int_{E} H\, \mathrm{d}\Gamma_{1} \oplus
{\scriptstyle (RL)}\int_{E} H\, \mathrm{d}\Gamma_{2}$.

\item[(\ref{T9}.b)]
 If $\Gamma_{1}\preceq \Gamma_{2}$, then \quad
${\scriptstyle (RL)}\int_{E} H\, \mathrm{d}\Gamma_{1} \preceq
{\scriptstyle (RL)}\int_{E} H\, \mathrm{d}\Gamma_{2}$.

\item[(\ref{T9}.c)] If $\Gamma_{1}\subseteq \Gamma_{2}$, then  \quad
${\scriptstyle (RL)}\int_{E} H\, \mathrm{d}\Gamma_{1} \subseteq
{\scriptstyle (RL)}\int_{E} H \, \mathrm{d}\Gamma_{2}$.

\item[(\ref{T9}.d)]
\begin{equation*}
d_H \Big(
{\scriptstyle (RL)}\int_{S}G \, \mathrm{d}\Gamma ,
{\scriptstyle (RL)}\int_{S} H \, \mathrm{d}\Gamma \Big)
\leq \sup_{s\in S} d_H (G(s), H(s))\cdot \overline{\Gamma}(S).
\end{equation*}
\end{description}
\end{theorem}
\subsection{Convergence results}
In the following we present some results of \cite{danilo,cgis2022,cgis2021} regarding convergent sequences of
Riemann-Lebesgue integrable interval-valued multifunctions.
Firstly we recall the definitions of convergence almost everywhere and convergence in measure for
interval-valued multimeasures.
\begin{definition}\label{D11}
\rm
 Let $\nu : \mathcal{C}\rightarrow [0, \infty)$ be a set function with $
\nu(\emptyset) = 0$, \mbox{$H:S\rightarrow L(\mathbb{R}_{0}^{+})$} a multifunction and a sequence of interval-valued multifunctions
$H_{n}:S \to L(\mathbb{R}_{0}^{+})$,  for every $n\in \mathbb{N}$.
It is said that:
\begin{description}
\item[(\ref{D11}.i)] $(H_{n})_n$ converges $\nu$-almost everywhere to $H$ on $S$ ( $H_{n}\overset{\nu-a.e.}{\longrightarrow} H$)
 if there exists $B\in \mathcal{C}$ with $\nu(B) = 0$ and $\lim\limits_{n\rightarrow \infty}d_H(H_{n}(s), H(s))=0,$ for every
 $ s\in S\setminus B.$
\item[(\ref{D11}.ii)] $(H_{n})$ $\nu$-converges to $H$ on $S$ ( $H_{n}\overset{\nu}{\longrightarrow} H$)
 if for every $\delta>0$,
 $B_{n}(\delta)= \{s\in S; d_H(H_{n}(s), H(s))\geq \delta\}\in \mathcal{C}$ and
$\lim\limits_{n\to \infty}\nu(B_{n}(\delta))=0$.
\end{description}
\end{definition}
\begin{theorem}
\mbox{\rm (\cite[Theorem 11]{cgis2022})}
  Let $\Gamma:\mathcal{C}\to L(\mathbb{R}_{0}^{+})$, $\G= [\nu_1, \nu_2],$ so that $\nu_{2}$ is of finite variation.
  Let $H=[h_{1}, h_{2}],\, H_{n}=[h_{1}^{(n)}, h_{2}^{(n)}]
:S\to L(\mathbb{R}_{0}^{+})$
%as in (\ref{H})
 be  multifunctions such that
$\sup \{ h_{2}(s), h_{2}^{(n)} (s), s \in S, \, n \in \mathbb{N} \} < +\infty$
and $H_{n} \overset{\widetilde{\Gamma}}{\rightarrow} H$. Then
  \[
\lim\limits_{n\rightarrow \infty}
d_H \Big(
{\scriptstyle (RL)}\int_{S} H_{n} \, \mathrm{d}\Gamma,
{\scriptstyle (RL)}\int_{S} H \, \mathrm{d}\Gamma \Big)=0.
\]
\end{theorem}
\begin{theorem}\label{4.13}
\mbox{\rm (\cite[Theorem 12]{cgis2022})}
  Suppose $\nu:\mathcal{C}\to [0, \infty)$ is monotone, of finite variation  and $\widetilde{\nu}$ satisfies \mbox{\rm \bf (E)}.
  Let $H=[h_{1}, h_{2}],\, H_{n}=[h_{1}^{(n)}, h_{2}^{(n)}]
:S\to L(\mathbb{R}_{0}^{+})$
%as in (\ref{H})
 be  multifunctions such that
$\sup \{ h_{2}(s), h_{2}^{(n)} (s), s \in S, \, n \in \mathbb{N} \} < +\infty$ and $H_{n} \overset{\nu-ae}{\rightarrow} H $, then
  \[\lim\limits_{n\rightarrow \infty} d_H \Big(
{\scriptstyle (RL)}\int_{S} H_{n}\, \mathrm{d}\nu,
{\scriptstyle (RL)}\int_{S} H\, \mathrm{d}\nu \Big)=0.\]
\end{theorem}
\begin{theorem} \label{4.14}
\mbox{\rm (\cite[Theorem 13]{cgis2022})}
  Let $\Gamma:=[\nu_1, \nu_2] :\mathcal{C}\to L(\mathbb{R}_{0}^{+})$  with $\nu_{1}$, $\nu_{2}$ monotone set functions
   satisfying \mbox{\rm \bf (E) } and $\nu_{2}$ of finite variation.\\  Let $H=[h_{1}, h_{2}],\, H_{n}=[h_{1}^{(n)}, h_{2}^{(n)}]
:S\to L(\mathbb{R}_{0}^{+})$
%as in (\ref{H})
 be  multifunctions such that
$\sup \{ h_{2}(s), h_{2}^{(n)} (s), s \in S, \, n \in \mathbb{N} \} < +\infty$
  and $H_{n} \overset{\widetilde{\Gamma}-ae}{\rightarrow} H $, then
  \[\lim\limits_{n\rightarrow \infty}
d_H \Big(
{\scriptstyle (RL)}\int_{S} H_{n}\, \mathrm{d}\Gamma,
{\scriptstyle (RL)}\int_{S} H\, \mathrm{d}\Gamma \Big)=0.\]
\end{theorem}
A Fatou type theorem for sequences of RL integrable interval-valued multifunctions holds.
\begin{theorem} \label{4.15}
\mbox{\rm (\cite[Theorem 14]{cgis2022})}
  Suppose $\nu:\mathcal{C}\to [0, \infty)$ is monotone with $0<\overline{\nu}(S)<\infty$ and $\widetilde{\nu}$ satisfies \mbox{\rm \bf (E)}. For every $n\in \mathbb{N}$, let
   $H_{n}=[h_{1}^{(n)}, h_{2}^{(n)}]$ be such that $(h_{2}^{(n)})_{n}$ is uniformly bounded. Then
  $$
{\scriptstyle (RL)}\int_{S} (\liminf_{n} H_{n})\, \mathrm{d}\nu \preceq
{\scriptstyle (RL)}\liminf_{n} \int_{S} H_{n}\, \mathrm{d}\nu.$$
\end{theorem}

In the sequel, some Lebesgue type theorems are presented.
\begin{theorem} {\rm (Monotone Convergence, \cite[Proposition 1]{cgis2021})}
Suppose $\Gamma = [\nu_1,\nu_2]$ with $\nu_i \in \mathscr{M}(S), \,\, i\in\{1,2\}$ of finite variation.
For every $n\in \mathbb{N}$, let $H_n =[h_1^{(n)}, h_2^{(n)}]$ be
%an interval-valued
a multifunction
%as in (\ref{H}),
such that  $(h_{2}^{(n)})$ is uniformly bounded and $H_n \preceq H_{n+1}\,$ for every $n \in \mathbb{N}$. Then
\[{\scriptstyle (RL)}\int_S \bigvee_n  H_n \, \mathrm{d}\Gamma =
\bigvee_n {\scriptstyle (RL)}\int_S H_n \, \mathrm{d}\Gamma.\]
\end{theorem}
It holds also a convergence type theorem for varying multisubmeasures.

\begin{theorem}\label{varying}
\mbox{\rm (\cite[Theorem 4.2]{cgis2021})}
Let $(H_n)_n:=([h_1^{(n)}, h_2^{(n)}])_n$ be a sequence of bounded
multifunctions,
and $(\Gamma_{n})_n:=([\nu^{(n)}_{1},\nu^{(n)}_{2}])_n$ a sequence of
multisubmeasures.
Suppose there exist  an interval-valued multisubmeasure $\Gamma := [\nu_1,\nu_2]$, with $\nu_2$
 of finite variation, and a bounded   multifunction $H :=[h_1,h_2]$
such that:
\begin{description}
\item[(\ref{varying}.a)] $H_{n} \preceq H_{n+1}$ for every $n \in \mathbb{N}$ and  $d_H (H_n, H) \to 0$ uniformly on $S$,
\item[(\ref{varying}.b)]  $\Gamma_n \preceq \Gamma_{n+1} \preceq \Gamma $ for every $n \in \mathbb{N}$  and $(\Gamma_n)_n$  setwise  converges  to  $\Gamma$
(namely $\lim_n \Gamma_n(A) = \Gamma(A)$ for every $A \in \mathcal{C}$).
\end{description}
Then
\[ \lim_{n \to \infty} d_H \Big(
{\scriptstyle (RL)}\int_{S}H_{n}\, \mathrm{d}\Gamma_n ,
{\scriptstyle (RL)}\int_{S} H \, \mathrm{d}\Gamma \Big) = 0.\]
\end{theorem}
\begin{proof}
By \ref{varying}.b) we have $\nu_i^{(n)} \leq \nu_i^{(n+1)} \leq \nu_i$ for every $i=1,2$ and for every $n \in \mathbb{N}$,
 moreover
$\lim_{n \to \infty} \overline{\nu}_i^{(n)} (A) = \overline{\nu}_i (A)$ for every $A \in \mathcal{C}$ and $i=1,2$.\\
Since $H_n, H$ are bounded and $\nu_2$ is of finite variation then, by  \cite[Proposition 1]{ccgis} and \cite[Proposition 2]{danilo},  $H_n,\, H \in RL^1_{\Gamma_k}(L(\mathbb{R}_0^+)) \cap RL^1_{\Gamma}(L(\mathbb{R}_0^+))$ for every $n,k \in \mathbb{N}$.
Let $\varepsilon > 0$ be fixed and let $n(\varepsilon)$ be such that $d_H(H_n,H) \leq \varepsilon$ for every $n \geq n(\varepsilon)$.
By \cite[Theorem 3]{danilo},
 for every $n \geq n(\varepsilon)$,
\begin{eqnarray*}
 && d_H \left( {\scriptstyle (RL)}\int_{S} H_n \, \mathrm{d}\Gamma_n,
{\scriptstyle (RL)}\int_{S} H \, \mathrm{d}\Gamma \right) \leq\\
&\leq& d_H \left( {\scriptstyle (RL)}\int_{S} H_n\, \mathrm{d}\Gamma_n,
{\scriptstyle (RL)}\int_{S} H \, \mathrm{d}\Gamma_n \right) +
d_H \left( {\scriptstyle (RL)}\int_{S} H\, \mathrm{d}\Gamma_n,
{\scriptstyle (RL)}\int_{S} H \, \mathrm{d}\Gamma \right) \leq\\
&\leq&
\varepsilon \overline{\Gamma}_n(S)
+
d_H \left( {\scriptstyle (RL)}\int_{S} H\, \mathrm{d}\Gamma_n,
{\scriptstyle (RL)}\int_{S} H \,\mathrm{d}\Gamma \right) \leq \\
&\leq&
\varepsilon \overline{\nu}_2(S) + d_H \left(
\left[ {\scriptstyle (RL)}\int_S h_1 \mathrm{d}\nu_1^{(n)},
{\scriptstyle (RL)} \int_S h_2 \mathrm{d}\nu_2^{(n)}\right],
 \left[ {\scriptstyle (RL)}\int_S h_1 \mathrm{d}\nu_1,
{\scriptstyle (RL)}\int_S h_2 \mathrm{d}\nu_2\right] \right).
\end{eqnarray*}
We have to evaluate
\begin{eqnarray*}
&& d_H \left( \left[
{\scriptstyle (RL)}\int_S h_1 \mathrm{d}\nu_1^{(n)},
{\scriptstyle (RL)}\int_S h_2 \mathrm{d}\nu_2^{(n)}\right],
 \left[
{\scriptstyle (RL)}\int_S h_1 \mathrm{d}\nu_1,
 {\scriptstyle (RL)}\int_S h_2 \mathrm{d}\nu_2\right] \right)=\\
&&=
\max_{i=1,2} \left\{
{\scriptstyle (RL)}\int_S h_i \mathrm{d}\nu_i -
{\scriptstyle (RL)}\int_S h_i \mathrm{d}\nu_i^{(n)} \right\}
\end{eqnarray*}
Using now \cite[Lemma  4.1]{cgis2021}
 the last term tends to 0 for $n \to \infty$ and so
\begin{eqnarray*}
&& \lim_{n \to \infty} d_H \left( {\scriptstyle (RL)}\int_{S} H_n \, \mathrm{d}\Gamma_n,
{\scriptstyle (RL)}\int_{S} H \, \mathrm{d}\Gamma \right)
 = 0.
\end{eqnarray*}

\end{proof}

Analogously to \cite[Remark 3]{danilo},
 Theorem \ref{varying} can be extended  to the bounded sequences $(H_n)_n$  converging $\overline{\Gamma}$-almost uniformly on $S$.
%namely for every
%$\varepsilon > 0$, there exists  $E_{\varepsilon} \in \mathcal{C}$    with
%$\overline{\Gamma} (E_{\varepsilon}) \leq \varepsilon\,  \mbox{  and   }\,
%d_H(H_n, H) \mbox{ converges uniformlyon } S \setminus E_{\varepsilon}.
%$ Moreover
\begin{corollary}\label{ext-varying}
\mbox{\rm (\cite[Corollary 1]{cgis2021})}
Let $(H_{n})_n := ([h_1^{(n)}, h_2^{(n)}])_n$ be a sequence of bounded  multifunctions
and $(\Gamma_{n})_n:=([\nu^{(n)}_{1}, \nu^{(n)}_{2}])_n$,  be a sequence of  multisubmeasures.
Suppose there exist
%an interval-valued
a multisubmeasure $\Gamma = [\nu_1, \nu_2]$ with $\nu_2$
of finite variation and a bounded  multifunction $H=[h_1,h_2] $
such that:
\begin{description}

\item[(\ref{ext-varying}.a)] $H_n \preceq H_{n+1}$ for every $n \in \mathbb{N}$ and  $d_H(H_n,H) \to 0$  $\Gamma$-almost uniformly on $S$,
\item[(\ref{ext-varying}.b)]  $\Gamma_n \preceq \Gamma_{n+1} \preceq \Gamma$, for every $ n\in \mathbb{N}$
 and $(\Gamma_n)$\, setwise converges  to  $\Gamma$.
\end{description}
Then
\[ \lim_{n \to \infty} d_H \Big(
{\scriptstyle (RL)}\int_{S}H_{n} \, \mathrm{d}\Gamma_n ,
{\scriptstyle (RL)}\int_{S} H\, \mathrm{d}\Gamma
\Big) = 0.\]
\end{corollary}
\begin{remark}
\rm
We can observe that the results of
 Theorem \ref{varying} and Corollary \ref{ext-varying} are still valid if we assume that $\Gamma_{n+1}
 \succeq  \Gamma_n \succeq \Gamma $ for every $n \in \mathbb{N}$, with the additional   hypothesis that
$ \sup_n \,  \overline{\Gamma}_n (S) < +\infty.$\\
Moreover, in Corollary \ref{ext-varying}, if $\Gamma_n = \Gamma = \nu$ in the condition (\ref{ext-varying}.a), then the monotonicity could be omitted.
\end{remark}
%==================0
\begin{comment}
We remind the convergence in measure, namely:
\begin{definition} \rm
It is said that the sequence $(H_{n})$ $\widetilde{\nu}$-converges to $H$
 ($H_{n} \overset{\widetilde{\nu}}{\rightarrow} H $\,) if $d_H(H_{n}, H)\overset{\widetilde{\nu}}\rightarrow 0$, i.e. for
\mg{\tiny \color{red} we need measurability, ma se iso la tilde non serve la C}
every $\delta>0$,
\[\lim\limits_{n\rightarrow \infty} \widetilde{\nu}(E_{n}(\delta))=0, \quad
\mbox{ where } \quad E_{n}(\delta)= \{s\in S: \,\, d_H(H_{n}(s), H(s)) \geq \delta \}\in \mathcal{C}.\]
\end{definition}
\end{comment}
%===================
In particular, in the finitely additive case, we obtain
\begin{theorem}
\mbox{\rm (\cite[Theorem 4.4]{cgis2021})}
Let $\nu:\mathcal{C}\to [0, \infty)$ be finitely additive and of finite variation.
  Let $H=[h_{1}, h_{2}],\, H_{n}=[h_{1}^{(n)}, h_{2}^{(n)}]
:S\to L(\mathbb{R}_{0}^{+})$
%as in (\ref{H})
 be  multifunctions such that
$\sup \{ h_{2}(s), h_{2}^{(n)} (s), s \in S, \, n \in \mathbb{N} \} < +\infty$
and $H_{n}\overset{\widetilde{\nu}}\rightarrow H$.
Then
\[\lim\limits_{n\rightarrow \infty} d_H \Big(
{\scriptstyle (RL)}\int_{S} H_{n}\, \mathrm{d}\nu,
{\scriptstyle (RL)}\int_{S} H\, \mathrm{d}\nu \Big)=0.\]
\end{theorem}
\subsection{Convergence results on atoms}
 Finally, the field of atoms in measure theory has many applications and has been studied by many authors (e.g., \cite{LMP,Pap2,PS}).\\
In order to obtain convergence results on atoms we
suppose $S$ is a locally compact Hausdorff topological space. We denote by $\mathcal{K}$ the lattice of all compact subsets of
$S$, $\mathcal{B}$ the Borel $\sigma$-algebra (i.e. the smallest $\sigma$-algebra containing $\mathcal{K}$) and $\mathcal{O}$
the class of all open sets.
\begin{definition}
\rm
 The set multifunction $\Gamma:\mathcal{B}\to L(\mathbb{R}_{0}^{+})$ is said to be regular if for every set $A\in \mathcal{B}$ and every $\varepsilon>0$
 there exist $K\in \mathcal{K}$ and $D\in \mathcal{O}$ such that $K\subseteq A\subseteq D$ and $\|\Gamma(D\setminus K)\|_{\mathcal{H}}<\varepsilon.$
\end{definition}
We observe that the regularity of  $\Gamma$ is equivalent to the regularity of $\nu_{2}$.
\begin{definition}
It is said that $B\in \mathcal{C}$ is an atom of an interval-valued multifunction $\Gamma:\mathcal{C}\to L(\mathbb{R}_{0}^{+})$ if $\{0\}\preceq \Gamma(B), \{0\}\neq \Gamma(B)$ and for every $C\in \mathcal{C}$, with
$C\subseteq B$, we have $\Gamma(C) = \{0\}$ or $\Gamma(B\setminus C)=\{0\}.$
\end{definition}

\begin{theorem}  \label{T1}
\mbox{\rm (\cite[Theorem 15]{cgis2022})}
Let $\Gamma:\mathcal{B}\to L(\mathbb{R}_{0}^{+})$ be
%an interval-valued
a  regular multisubmeasure
%as in (\ref{M}),
 of finite variation and satisfying property $(\boldsymbol{\sigma})$ and let $H:S\to L(\mathbb{R}_{0}^{+})$ be bounded.
If $B\in \mathcal{B}$ is an atom of \, $\Gamma$, then
${\scriptstyle (RL)}\int_{B}H\, \mathrm{d}\Gamma= H(b)\smallbullet \Gamma(\{b\}),$ where $b\in B$
is the single point resulting by \mbox{\rm  \cite[Corollary 4.7]{LMP}}.
\end{theorem}
\begin{proof} Firstly, we prove  the uniqueness of  $b\in B$.
 Because $\Gamma$ is an interval-valued regular multisubmeasure, then the set functions $\nu_1$ and $\nu_2$
are null-additive an regular too.
 Suppose $B\in \mathcal{B}$ is an atom of $\Gamma$. Then, $B$ is an atom of $\nu_1$ and $\nu_2$.
According to~\cite[Corollary 4.7]{LMP},  for $\nu_{i}, i\in \{1,2\},$ there exists a unique point $b_{i}, i\in \{1,2\}$
such that $\nu_{i}(\{b_{i}\})=\nu_{i}(B)$ and $\nu_{i}(B\setminus \{b_{i}\})=0$, for $i\in \{1,2\}$.
We prove that $b_{1}=b_{2}$. If it is not true
 then $\{b_{1}\}\subset B\setminus\{b_{2}\}$. By the monotonicity of $\nu_2$ we have
 $\nu_2(\{b_{1}\})\leq \nu_2(B\setminus\{b_{2}\})=0$. Since  $\nu_1\leq \nu_2$ then $\nu_1(\{b_{1}\})=0$,
but $\nu_{1}(\{b_{1}\})=\nu_1(B)>0$, and  we have a contradiction. Therefore, there is only one point
$b\in B$ such that $\nu_{i}(\{b\})=m_{i}(B)$ and $\nu_{i}(B\setminus \{b\})=0$, for $i\in \{1,2\}$.

By the $RL_{\Gamma}$-integrability of $H$, then $h_1$ is $RL_{\nu_1}$-integrable and $h_2$ is $RL_{\nu_2}$-integrable.
According to \cite[Theorem 11]{ccgis} and Subsection \ref{sec-comp}, $h_1, h_2$ are Gould integrable in the sense of \cite{G}, and moreover:
$${\scriptstyle (RL_{\nu_1})}\int_{B} h_1  \mathrm{d} \nu_1 =
(G)\int_{B} h_1  \mathrm{d}\nu_1,
\quad {\scriptstyle (RL_{\nu_2})}\int_{B} h_2  \mathrm{d}\nu_2 =
(G)\int_{B}h_2  \mathrm{d}\nu_2,$$
 where $(G)\int_{B} h_1  \mathrm{d}\nu_1$,
 $(G)\int_{B}h_2  \mathrm{d}\nu_2$ are the Gould integrals of $h_1, \, h_2$  respectively.
~Applying now~\cite[Theorem 3]{ccgs16} and \cite[ Remark 5]{cgis2022}, we have
$$(RL_{\Gamma})\int_{B}H  \mathrm{d}\Gamma=
(G)\int_{B}H \mathrm{d}\Gamma=H(b)\Gamma(\{b\}).$$
\end{proof}

\begin{theorem}
\mbox{\rm (\cite[Theorem 16]{cgis2022})}
Let $\Gamma:\mathcal{B}\to L(\mathbb{R}_{0}^{+})$ be
a regular  multisubmeasure  of finite variation and satisfying property $(\boldsymbol{\sigma})$.
 Let $H:S\to L(\mathbb{R}_{0}^{+})$ be bounded
and, for every $n\in \mathbb{N},$ let $H_{n}=[u_{n}, v_{n}]$ be such that $(v_{n})_{n}$ is uniformly bounded.
If $B\in \mathcal{B}$ is an atom of $\Gamma$ and $H_{n}(b)\overset{d_H}\longrightarrow H(b)$, where $b\in B$ is the single point
resulting by Theorem \mbox{\rm \ref{T1}}, then
\[\lim\limits_{n\rightarrow \infty} d_H \Big(
{\scriptstyle (RL)}\int_{B} H_{n}\, \mathrm{d}\Gamma,
{\scriptstyle (RL)}\int_{B} H\, \mathrm{d}\Gamma \Big)=0.\]
\end{theorem}
\begin{proof}
 By Theorem \ref{T1}, there exists a unique point $b\in B$ such that: $$\Gamma(B\setminus \{b\})=\{0\}, \quad
{\scriptstyle (RL)}\int_{B} Hd\Gamma=H(b)\cdot \Gamma(B).$$

Similarly, for every $n\in \mathbb{N}$,
 there is a unique $b_{n}\in B$ such that:
$$\G(B\setminus \{b_{n}\})=\{0\}, \quad
{\scriptstyle (RL)}\int_{B} H_{n}\mathrm{d}\Gamma=H_{n}(b_{n})\cdot \Gamma(B).$$

 If there exists $n_{0}\in \mathbb{N}$
 such that $b_{n_{0}}\neq b$, this means that $\{b_{n_{0}}\}\subset B\setminus \{b\}$, and by the monotonicity
 of $\G$, it follows that: $\G(\{b_{n_{0}}\})\preceq \G(B\setminus \{b\})=\{0\};$ however, this is not possible since $\G(\{b_{n_{0}}\})=\G(B)\neq \{0\}.$
 Therefore, for every $n\in \mathbb{N}$, $b_{n}=b.$
 Then:
 \begin{equation*}
 d_{H} \left({\scriptstyle (RL_{\Gamma})}\int_{B} H_{n}\mathrm{d}\Gamma,{\scriptstyle (RL_{\Gamma})}
\int_{B} H \mathrm{d}\Gamma \right)\leq d_{H} (H_{n}(b),H(b))\cdot\overline{\Gamma}(B)\longrightarrow 0,\quad \mbox{ for } n\rightarrow \infty.
 \end{equation*}
\end{proof}
%=============================================
{\bf Acknowledgement and funding.}
This study was partly funded by the Unione europea - Next Generation EU, Missione 4 Componente C2 - CUP Master: J53D2300390 0006, CUP: J53D23003920 006 - Research project of MUR (Italian Ministry of University and Research) PRIN 2022  “Nonlinear differential problems with applications to real phenomena” (Grant Number: 2022ZXZTN2).
\\
The last  author is member of the ``Gruppo Nazionale per l'Analisi Matematica, la Probabilità e le loro Applicazioni'' (GNAMPA) of the Istituto Nazionale di Alta Matematica (INdAM) and of GRUPPO DI LAVORO UMI - Teoria dell'Approssimazione e Applicazioni - T.A.A.

%=============================================

\Addresses
\end{document}